%% file: On_Positive_Braids,_Monodromy_Groups_and_Framings.tex
\documentclass[11pt,a4paper]{article}

\usepackage[utf8]{inputenc}  
\usepackage[T1]{fontenc}   
\usepackage[english]{babel}

\usepackage{amsmath}
\usepackage{amsfonts}
\usepackage{amssymb}
\usepackage{theoremref}
\usepackage{mathrsfs}
\usepackage{amsthm}
\usepackage{mathtools}

\usepackage{enumitem}
\setlist[enumerate, 1]{label=\Roman*.}
\setlist[enumerate, 2]{label*=\Alph*.}
\setlist[enumerate, 3]{label*=\alph*.}
\setlist[enumerate, 4]{label*=\arabic*}

\usepackage{varioref}      
\usepackage{hyperref} 

\usepackage{import}
\usepackage{color}
\usepackage{calc}

\usepackage{graphicx}
\graphicspath{{./Figures/3braids/}}

\usepackage{biblatex}
\newtheorem{thm}{Theorem}
\newtheorem{lem}{Lemma}[section]
\newtheorem{prop}{Proposition}[section]
\newtheorem{cor}[thm]{Corollary}

\theoremstyle{definition}
\newtheorem{defn}{Definition}[section]

\newtheorem{exmp}{Example}[section]

\theoremstyle{remark}
\newtheorem{rem}{Remark}[section]

\title{\normalsize\textbf{ON POSITIVE BRAIDS, MONODROMY GROUPS AND FRAMINGS}}
\date{}
\author{Livio Ferretti}

\begin{document}
	
\maketitle

\begin{abstract}
		We associate to every positive braid a group, generalizing the geometric monodromy group of an isolated plane curve singularity. If the closure of the braid is a knot, we identify the corresponding group with a framed mapping class group. In particular, this gives a well defined knot invariant. As an application, we obtain that the geometric monodromy group of an irreducible singularity is determined by the genus and the Arf invariant of the associated knot.
\end{abstract}

\tableofcontents
	
\section{Introduction}\label{section_introduction}

Singularity theory is a genuine source of examples and inspiration for knot theory. Since the topological type of an isolated plane curve singularity is determined by an associated link, it is possible to understand properties of the singularity from a knot theoretical viewpoint, and knot theory has been successfully applied to solve algebraic questions. In another direction, links of singularities form an interesting class of links, with special properties and invariants that follow from the whole machinery of singularity theory. It is often unclear which of those properties are inherently algebraic and which ones could be generalized to wider classes of knots and links. Among other invariants, the fundamental group of the discriminant complement and the geometric monodromy group have drawn much attention but have proved to be hard to investigate.

In \cite{baader_secondary_2021}, Baader and Lönne associate to any positive braid an abstract group defined by generators and relations, which they call secondary braid group. Their motivation comes from the similarities between the combinatorial structure of positive braids and that of isolated plane curve singularities. In particular, they prove that for braids of type $ADE$ and for braids of minimal braid index whose closure is a torus link $T_{p,q}$ the secondary braid group is isomorphic to the fundamental group of the discriminant complement of the corresponding singularities (simple singularities in the former case, Brieskorn-Pham singularities $f(x,y) = x^p+y^q$ in the latter, see \cite{lonne_fundamental_2007}). However, because of difficulties in dealing with conjugation in the positive braid monoid, they can prove that the secondary braid group is a well defined link invariant only for positive braids whose closure contains a positive half twist.

Inspired by their construction and in analogy with the definition of the geometric monodromy group of a singularity, in this article we associate to any positive braid $\beta$ a group $\mathit{MG}(\beta)$, which we call the monodromy group of the positive braid, defined as a subgroup of the mapping class group of the unique genus minimizing Seifert surface of the closure $\hat{\beta}$, generated by the Dehn twists around some natural family of curves. The monodromy group of a positive braid is a quotient of Baader's and Lönne's secondary braid group which contains the monodromy diffeomorphism of the positive braid. Moreover, it is a generalization of the geometric monodromy group of an isolated plane curve singularity to the setting of positive braids.

\begin{thm}\label{theorem_group_singularity}
	Let $f:\mathbb{C}^2 \rightarrow \mathbb{C}$ define an isolated plane curve singularity and $L(f)$ be the link of $f$. Then there exists a positive braid $\beta$ representing $L(f)$ such that the geometric monodromy group of $f$ is equal to $\mathit{MG}(\beta)$.
\end{thm}

In \cite{portilla_cuadrado_vanishing_2021}, Cuadrado and Salter proved that the geometric monodromy group of any singularity of genus at least $5$ and not of type $A_n$ and $D_n$ is a framed mapping class group, i.e. the stabilizer of some canonical framing on the Milnor fibre associated to the singularity and, among other things, they use this result for deducing the non-injectivity of the geometric monodromy representation. Following their approach, the main result of this paper is an identification of the monodromy group of a positive braid $\beta$ whose closure is a knot with a framed mapping class group on the genus minimizing surface $\Sigma_{\beta}$.


\begin{thm}\label{theorem_framed}
	Let $\beta$ be a prime positive braid not of type $A_n$ and whose closure is a knot. Then, for all but finitely many such braids, there exists a framing $\phi_{\beta}$ on $\Sigma_{\beta}$ such that the monodromy group $\mathit{MG}(\beta)$ is equal to the framed mapping class group $\mathrm{Mod}(\Sigma_{\beta},\phi_{\beta})$.
\end{thm}

For the definition of a positive braid of type $A_n$, see Section~\ref{section_definition}. It is important to mention that the infinite family of braids of type $A_n$ that we exclude from Theorem~\ref{theorem_framed} is in fact one of the only cases where the monodromy group was already explicitly known: it is isomorphic to the Artin group of the corresponding type \cite{perron_groupe_1996}. Those groups are not isomorphic to any framed mapping class group, so their exclusion is a necessity, rather than a limitation of any sort.

Of course, as a consequence of Theorems~\ref{theorem_group_singularity} and~\ref{theorem_framed}, in the restricted context of singularities we immediately obtain that the geometric monodromy group of an irreducible singularity is controlled by a framing. In fact, as explained in Remark~\ref{remark_links_sing}, one can see that our proof of Theorem~\ref{theorem_framed} also applies to many links, including links of singularities not of type $A_n$ and $D_n$, thus recovering the results of \cite{portilla_cuadrado_vanishing_2021} up to finitely many exceptions. On the other hand, there are some infinite families of positive braid links for which our methods do not seem to work, see Remark~\ref{remark_links}. In spite of the increased combinatorial difficulty, working in the more general setting of positive braids has some advantages, as we will now explain.

Since the topological type of a singularity is completely determined by its link, a priori every topological invariant of a singularity should be somehow readable from the link. For instance, the Milnor number corresponds to the minimal first Betti number, while the multiplicity corresponds to the braid index \cite{williams_braid_1988}. However, this translation is often far from straightforward. Now, it turns out that framed mapping class groups are determined by the value of the framing on the boundary components of the surface and a certain Arf invariant associated to the framing. In the case of a surface $\Sigma$ with connected boundary, the value of the framing on the boundary is always equal to the Euler characteristic of $\Sigma$, so that the framed mapping class group is determined simply by the genus of $\Sigma$ and the Arf invariant of the framing. Working with positive braids, we are able to identify the Arf invariant of the framing with the classical Arf invariant of the boundary knot. We thus obtain the following corollaries, expressing the geometric monodromy group of an irreducible singularity in terms of well known invariants of its knot.


\begin{cor}\label{cor_knots}
	Let $\beta$ be a prime positive braid not of type $A_n$ and whose closure is a knot $K$. Up to finitely many exceptions, the monodromy group of $\beta$ is an invariant of $K$, determined by its genus and Arf invariant.
\end{cor}

\begin{cor}\label{cor_singularities}
	Let $f$ define an irreducible isolated plane curve singularity that is not of type $A_n$ and $K(f)$ be the knot of the singularity. For all but finitely many such singularities, the geometric monodromy group of $f$ is determined by the genus and the Arf invariant of $K(f)$.
\end{cor}

It is important to point out that the monodromy group of a positive braid is proved to be an invariant of the braid closure only if the latter is connected; for braids whose closure is disconnected, the strongest invariance result is Corollary~\ref{cor_conj_invariance}.
 
From a purely knot theoretical viewpoint, Theorem~\ref{theorem_framed} might seem disappointing. It implies that, if the closure of a positive braid is a knot (up to finitely many exceptions), its monodromy group is an invariant of the knot, but a rather useless one: it is hard to compute, but determined by two classical and much easier invariants, a natural number and a mod $2$ class. Its interest lies in negative results such as Corollary~\ref{cor_singularities}. The geometric monodromy group, which was typically considered a rich yet hard to investigate invariant of a plane curve singularity, turns out, in the case of irreducible singularities, to be determined by two simple knot invariants, and the question whether two irreducible singularities have the same geometric monodromy group can be answered by a direct and easy computation, using existing formulas for the Arf invariant of a knot. Of course, for each fixed genus there are many different irreducible singularity, so there will be different singularities with the same geometric monodromy group. We believe that for big enough genus both values of the Arf invariant are realized, so that there would be exactly two geometric monodromy groups.

The study of the monodromy group of a positive braid has naturally its place in the context of finitely generated subgroups of the mapping class group, and in particular subgroups generated by Dehn twists around a family of curves with prescribed intersection pattern. Those subgroups are interesting by themselves from a mapping class group theoretical viewpoint, but also appear naturally in different contexts, such as singularity theory or in the study of Lefschetz fibrations. The question of what groups can arise in this way is completely solved in the case of two Dehn twists, see for example Chapter $3$ of \cite{farb_primer_2012}, but is in general widely open. In \cite{perron_groupe_1996} Perron and Vannier, interested in the geometric monodromy group of singularities, proved that if the intersection pattern of the curves is a Dynkin diagram of type $A_n$ or $D_n$, the group generated by the Dehn twists is isomorphic to the Artin group of corresponding type, and conjectured this to be true for general graphs. This was later disproved by Labruère \cite{labruere_generalized_1997} and Wajnryb \cite{wajnryb_artin_1999}, whose results show that the only Artin groups whose Dynkin diagram is a tree and that geometrically embed in the mapping class group are precisely the ones of type $A_n$ and $D_n$. Notice that, contrary to what Wajnryb claimed, the Artin groups of type $\tilde{A}_n$, i.e. whose Dynkin diagram is a cycle, do geometrically embed in the mapping class group, as recently proved by Ryffel in \cite{ryffel_curves_2023}. The theory of framed mapping class groups seems to suggest that, at least if the intersection pattern is in some sense rich enough, those finitely generated subgroups are controlled by a framing on the surface. Theorem~\ref{theorem_framed} is an example of such a result.

As a final remark, although in this paper we concentrate only on positive braids, they are not the only natural class of links generalizing links of singularities to which one could try to associate a monodromy group. A'Campo's divide links form another such interesting family, see Section~\ref{section_singularities}. More generally, it is known that the Milnor fibre of an isolated plane curve singularity can be constructed by a sequence of positive Hopf plumbings such that the core curves of the Hopf bands coincide with a distinguished basis of vanishing cycles of the singularity, the Dehn twist around which generate the geometric monodromy group, see \cite{ishikawa_plumbing_2002}. As explained in Remark~\ref{rem_plumbings}, this is also the case for the monodromy group of a positive braid. Going one step further, for a general sequence of positive Hopf plumbings, one could define a monodromy group as the group generated by all the Dehn twists around the core curves of the Hopf bands. We expect that, at least for knots, results similar to Theorem~\ref{theorem_framed} should hold in this more general setting. This is not difficult to see for Hopf plumbings with intersection pattern a tree and whose boundary is a knot of sufficiently big genus.

\paragraph{Structure of the paper:} In Section~\ref{section_definition} we define the monodromy group of a positive braid and prove some basic invariance properties. In Section~\ref{section_singularities} we recall some basics of singularity theory and, using A'Campo's theory of divides, we prove Theorem~\ref{theorem_group_singularity}. In Section~\ref{section_framings} we discuss the general theory of framed mapping class groups and construct the framing appearing in Theorem~\ref{theorem_framed}. Finally, Section~\ref{section_proof} is the technical part of the paper, in which we prove Theorem~\ref{theorem_framed}. This basically consists of a lengthy case distinction that allows us to apply general results about framed mapping class groups.

\paragraph{Acknowledgements:} I wish to thank Sebastian Baader for suggesting the topic and guiding me through this project. I am also very grateful to Livio Liechti for the several interesting discussions, and in particular for pointing out the connection to framed mapping class groups. Finally, thanks to Nick Salter, Pablo Portilla Cuadrado and Michael Lönne for their interesting comments, and to the anonymous referee for the useful suggestions that greatly improved the exposition.

\section{The monodromy group of a positive braid}\label{section_definition}

Let $B_N^+$ be the monoid of positive braids on $N$ strands and $\beta \in B_N^+$. We will usually represent such a braid with a \textit{brick diagram}, a plane graph with $N$ vertical lines connected by horizontal segments corresponding to the crossings. Since all the crossings are positive, one can reconstruct the braid from the brick diagram. It is well known that, if $\beta$ is non-split, its closure $\hat{\beta}$ is a fibred link, whose fibre surface can be constructed by taking a disk for each strand of $\beta$ and, for each generator $\sigma_i$ in $\beta$, gluing a half-twisted band between the $i$-th and $(i+1)$-th disks. The brick diagram of $\beta$ naturally embeds in this surface as a retract. Let us denote this fibre surface by $\Sigma_\beta$, and let $g$ be its genus and $r$ the number of boundary components. On $\Sigma_\beta$ there is a standard family of $2g+r-1$ curves $\gamma_i$, oriented counterclockwise, which are in one-to-one correspondence with the \textit{bricks}, i.e. the innermost rectangles, of the brick diagram of $\beta$ and form a basis of the first homology of $\Sigma_\beta$. See Figure~\ref{fibre surface} for an example of $\Sigma_\beta$ with the corresponding curves for $\beta = \sigma_3\sigma_1\sigma_2\sigma_1^2\sigma_3\sigma_2$. The intersection pattern of those standard curves can be read off directly from the brick diagram, in the so called \textit{linking graph}:

\begin{defn}
	Let $\beta$ be a positive braid word. Its \textit{linking graph} is a graph whose vertices are the bricks of the brick diagram of $\beta$; two vertices are connected by an edge if and only if the corresponding bricks are arranged as the two bricks of the braids $\sigma_i^3$, $\sigma_i\sigma_{i+1}\sigma_i\sigma_{i+1}$ or $\sigma_{i+1}\sigma_i\sigma_{i+1}\sigma_i$.
\end{defn}

Notice that two vertices of the linking graph are connected with an edge if and only if the corresponding curves intersect each other. Linking graphs of positive braids were studied in great detail in \cite{baader_checkerboard_2018}. Here it is worth mentioning that since positive braid links are visually prime by \cite{cromwell_positive_1993}, a positive braid link is prime if and only if the linking graph is connected. In what follows, we will say that a positive braid link is of \textit{type $A_n$} (resp. $D_n$) if it isotopic to the closure of the braid $\sigma_1^{n+1}$ (resp. $\sigma_1^{n-2}\sigma_2\sigma_1^2\sigma_2$). Those braids have as linking graph the simply laced Dynkin diagram of type $A_n$ or $D_n$.

\begin{figure}
	\centering
	\def\svgscale{0.3}
	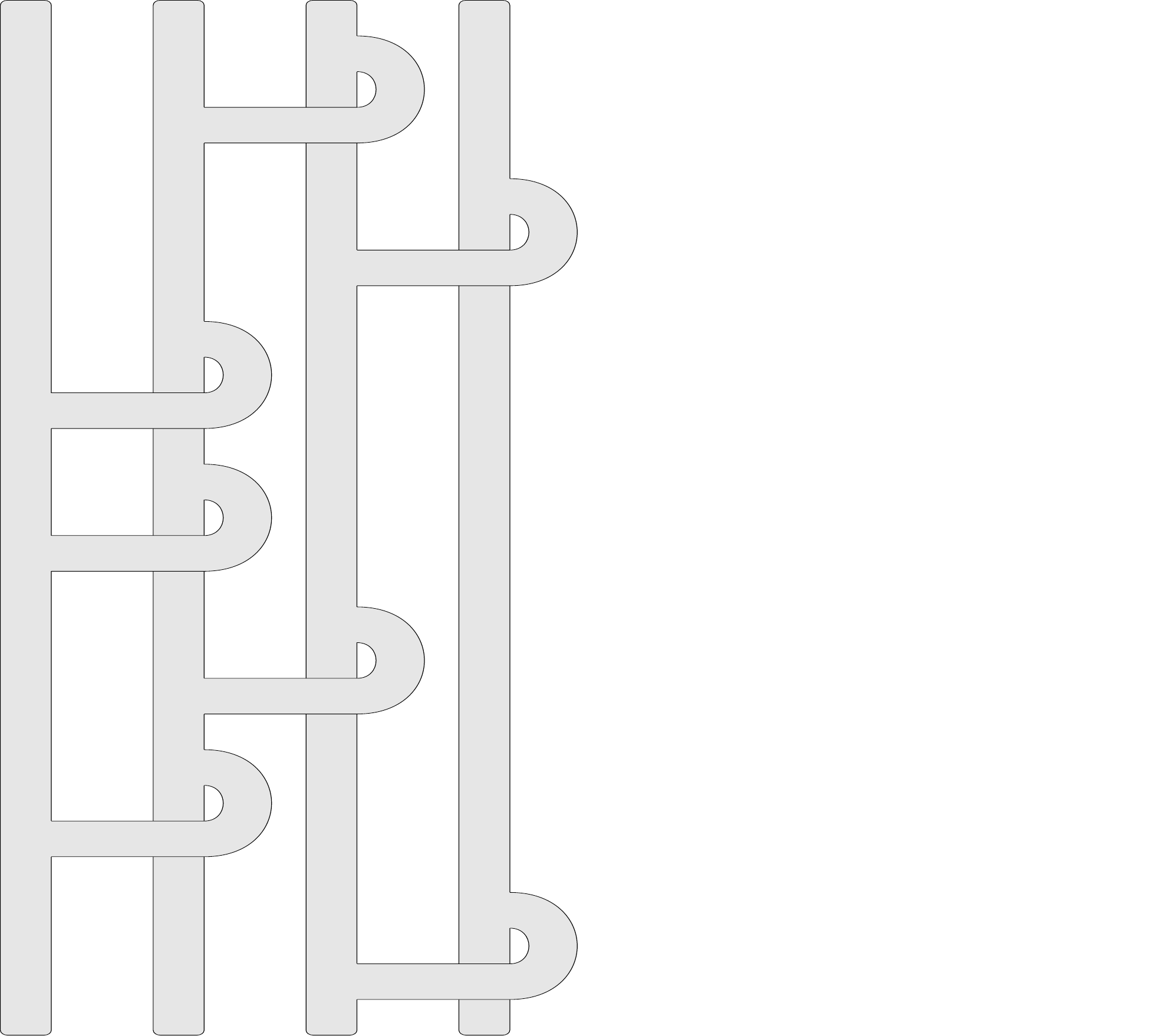
	\caption{The fibre surface of $\sigma_3\sigma_1\sigma_2\sigma_1^2\sigma_3\sigma_2$, its brick diagram and the corresponding linking graph.}
	\label{fibre surface}
\end{figure}

\begin{defn}
	Let $\beta$ be a positive braid. The \textit{monodromy group} $\mathit{MG}(\beta)$ is the subgroup of the mapping class group of $\Sigma_\beta$ generated by all the Dehn twists around the curves $\gamma_i, i=1,\cdots,2g+r-1$, i.e. $$\mathit{MG}(\beta) = \langle T_{\gamma_1}, \cdots, T_{\gamma_{2g+r-1}}\rangle \leqslant \mathrm{Mod}(\Sigma_\beta).$$
\end{defn}

\begin{rem}\label{rem_plumbings}
	As we just said, if a positive braid $\beta$ is non-split, then its closure is fibred, and $\Sigma_{\beta}$ is the fibre surface. In fact, this surface can be constructed by a sequence of plumbings of positive Hopf bands, and the curves $\gamma_i$ are precisely the core curves of those Hopf bands. The monodromy group of $\beta$ therefore somehow reflects this plumbing structure.
\end{rem}

\begin{exmp}
	As already mentioned, it follows from \cite{perron_groupe_1996} that if $\beta = \sigma_1^{n+1}$ then $\mathit{MG}(\beta)$ is isomorphic to the Artin group of type $A_n$. Similarly, for $\beta = \sigma_1^{n-2}\sigma_2\sigma_1^2\sigma_2$, $\mathit{MG}(\beta)$ is isomorphic to the Artin group of type $D_n$
\end{exmp}

From the definition, it is clear that $\mathit{MG}(\beta)$ is invariant under far-commutativity (i.e. $\sigma_i\sigma_j = \sigma_j\sigma_i$ for $|i-j|\geq 2$) and positive Markov move.

\medskip 

\begin{prop}[Elementary conjugation invariance]\label{prop_elem_conj}
	Let $\beta$ be a positive braid on $N$ strands. Then for all $1\leq i\leq N-1$, $\mathit{MG}(\beta\sigma_i) \simeq \mathit{MG}(\sigma_i\beta)$.
\end{prop}

\begin{proof}
	Consider the fibre surfaces of $\beta\sigma_i$ and $\sigma_i\beta$. Those surfaces are isotopic, by sliding the topmost band between the $i$-th and $(i+1)$-th disks along the back of the disks and bringing it in the lowermost position. Note that this isotopy restricts to the identity outside of the $i_{th}$-column. The surfaces $\Sigma_{\beta\sigma_i}$ and $\Sigma_{\sigma_i\beta}$ can hence be schematically represented as in Figure~\ref{elem_conj} , where we drew the $i_{th}-$column and the light grey boxes on the two sides represent the remaining parts of the surface.
	
	\begin{figure}
		\centering
		\def\svgscale{0.5}
		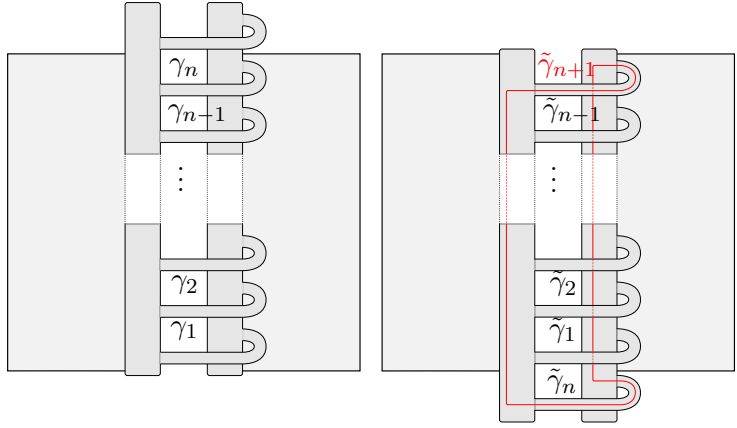
		\caption{The isotopy between $\Sigma_{\beta\sigma_i}$ and $\Sigma_{\sigma_i\beta}$}
		\label{elem_conj}
	\end{figure}

	Let us number the standard curves of the $i_{th}$ column as in Figure~\ref{elem_conj}. The isotopy will send each $\gamma_i, i=1,\cdots,n-1$ to the corresponding $\tilde{\gamma}_i, i=1,\cdots,n-1$ and transform $\gamma_n$ into the red curve $\tilde{\gamma}_{n+1}$. All what we have to prove is then that we can generate the Dehn twists around the curves $\tilde{\gamma}_1,\cdots,\tilde{\gamma}_{n-1},\tilde{\gamma}_{n+1}$ using $\tilde{\gamma}_1,\cdots,\tilde{\gamma}_{n-1},\tilde{\gamma}_n$, and vice-versa. But we note that $$\tilde{\gamma}_{n+1} = T_{\tilde{\gamma}_{n-1}}\cdots T_{\tilde{\gamma}_2} T_{\tilde{\gamma}_1} (\tilde{\gamma}_n),$$ so that for $h = T_{\tilde{\gamma}_{n-1}}\cdots T_{\tilde{\gamma}_2} T_{\tilde{\gamma}_1}$ we have $$T_{\tilde{\gamma}_{n+1}}= h T_{\tilde{\gamma}_n} h^{-1}$$ and the result is proved.

\end{proof}

\medskip

\begin{prop}[Braid relation invariance]\label{prop_braid_rel_invariance}
	Let $\alpha$ and $\beta$ be two positive braids related by a braid relation, then $\mathit{MG}(\alpha) \simeq \mathit{MG}(\beta)$.
\end{prop}

\medskip

\begin{proof}
	Up to elementary conjugation, we can suppose that $\alpha = \omega \sigma_i\sigma_{i+1}\sigma_i$ and $\beta = \omega\sigma_{i+1}\sigma_i\sigma_{i+1}$, where $\omega$ is a positive braid on $N$ strands and $1\leq i\leq N-2$. At the level of surfaces $\Sigma_\alpha$ and $\Sigma_\beta$ the braid relation can be realized by an isotopy as in Figure~\ref{braid_rel}. It is clear that all the standard curves $\gamma_i$ are fixed by this isotopy but the ones (at most two) passing through the slidden band.
		
	\begin{figure}
		\centering
		\def\svgwidth{\columnwidth}
		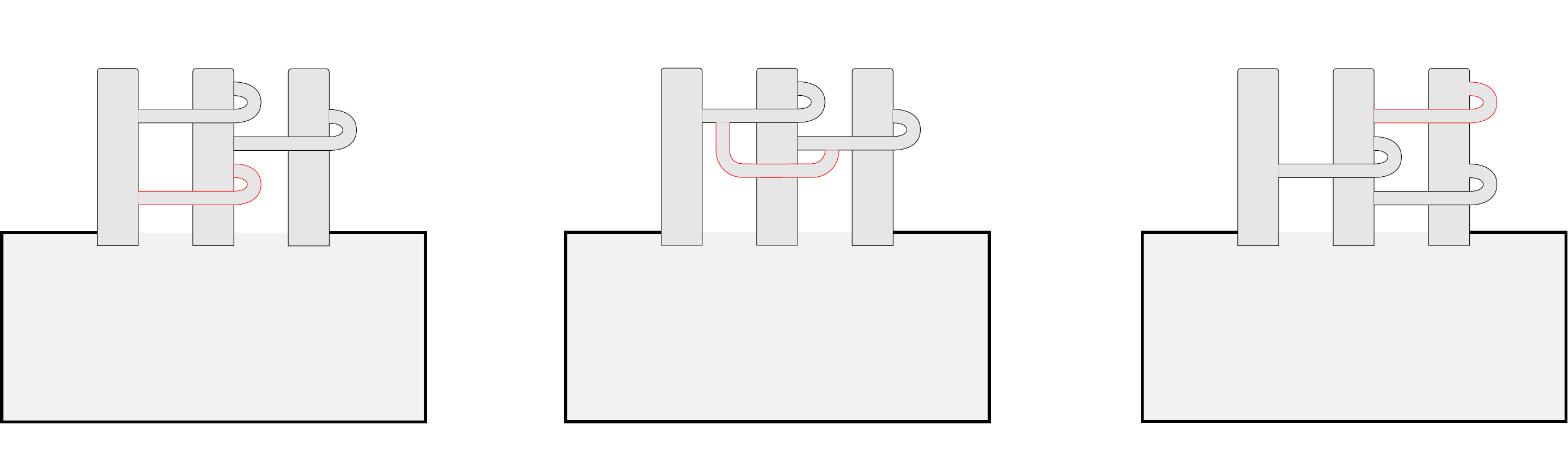
		\caption{The isotopy between $\Sigma_{\alpha}$ and $\Sigma_{\beta}$}
		\label{braid_rel}
	\end{figure}

	\begin{itemize}
		\item There is a generator $\sigma_i$ in $\omega$: in this case, there are two curves on $\Sigma_{\alpha}$ which are modified by the isotopy. Let us call them $\gamma_1$ and $\gamma_2$, as in Figure~\ref{braid_rel_curve}. We see that, after the isotopy, $\gamma_1$ is transformed into the corresponding $\tilde{\gamma}_1$, while $\gamma_2$ becomes $T_{\tilde{\gamma}_1}^{-1}(\tilde{\gamma}_2)$. All the other standard curves are fixed. Therefore, we get that $\mathit{MG}(\alpha) \simeq \mathit{MG}(\beta)$.
		
		\item There is no $\sigma_i$ in $\omega$: in this case, the only curve modified by the isotopy is $\gamma_1$, which as before is transformed into $\tilde{\gamma}_1$. Again, we directly have that $\mathit{MG}(\alpha) \simeq \mathit{MG}(\beta)$.
		
		\begin{figure}
			\centering
			\def\svgwidth{\columnwidth}
			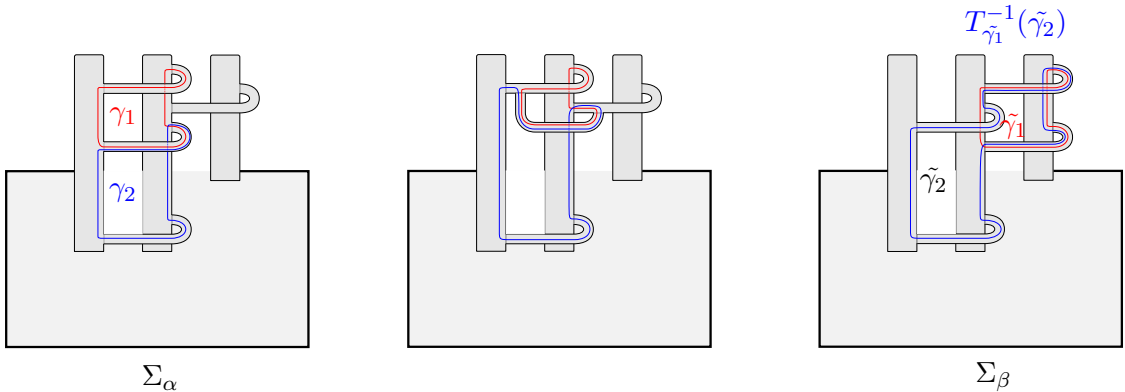
			\caption{First case of braid relation invariance}
			\label{braid_rel_curve}
		\end{figure}
		
	\end{itemize}
	
\end{proof}

The following corollary now follows directly by an observation of Orevkov about Garside's solution of the conjugacy problem in the braid group, saying that, in the presence of a positive half twist, two conjugate positive braids can be related by a sequence of braid relations and elementary conjugations, see \cite[Section 6]{baader_secondary_2021}.

\begin{cor}\label{cor_conj_invariance}
	Let $\alpha$ and $\beta$ be positive braids such that the closures are braid isotopic and contain a positive half twist, then $\mathit{MG}(\alpha) \simeq \mathit{MG}(\beta)$.
\end{cor} 

\section{Divides and monodromy of singularities}\label{section_singularities}

The monodromy group of a positive braid is a generalization of the geometric monodromy group of an isolated plane curve singularity. In this section, we will make this statement more precise.

Let $f: \mathbb{C}^2 \rightarrow \mathbb{C}$ define an isolated plane curve singularity. For a suitably small radius $r > 0$, the sphere $\partial (B^4_r) \subset \mathbb{C}^2$ intersects the singular curve $C = f^{-1}(0)$ transversally, so that the intersection $L(f) = C \cap \partial (B^4_r)$ is a link in $S^3 = \partial (B^4_r)$, called the link of the singularity. It is well known that the isotopy type of $L(f)$ completely determines the topological type of the singularity. Moreover, in \cite{milnor_singular_1968} Milnor proved that the map $\frac{f}{|f|}: S^3\setminus L(f) \rightarrow S^1$ is a fibration. Singularity links are therefore fibred links, with fibre a surface $\Sigma(f)$ called the \textit{Milnor fibre}. It turns out that all the singularity links are iterated torus links, and in particular positive braid links. The fibration induces a monodromy diffeomorphism of the fibre, which is only defined up to isotopy and therefore defines a mapping class in $\mathrm{Mod}(\Sigma(f))$, called the \textit{geometric monodromy} of the singularity. The geometric monodromy is an important invariant, which determines the topology of the singularity and has been intensively studied in the context of singularity theory.

By the study of the deformations of the singularity, the geometric monodromy can be "promoted" to the so called \textit{geometric monodromy group} of the singularity. It is a subgroup of $\mathrm{Mod}(\Sigma(f))$ generated by the Dehn twists around some specific curves on the Milnor fibre, called \textit{vanishing cycles}. The geometric monodromy can be expressed as a product of those generators and is therefore an element of the geometric monodromy group. We will not discuss the original definition of the geometric monodromy group of a singularity since, although classic, it would require quite some background knowledge in singularity theory and will not be useful for our purposes. However, there exists an easy combinatorial model for the Milnor fiber of a singularity which allows us to directly define the geometric monodromy group in terms of explicit generators. This was constructed by A'Campo using the theory of divides.

\begin{defn}
	A \textit{divide} $\mathcal{D}$ is a generic relative immersion of finitely many intervals in the unit disk $(D^2,\partial(D^2))$.
\end{defn}

Here, \textit{generic} means that the only singularities are double points and that the intervals meet the boundary $\partial(D^2)$ transversally. Examples of divides can be seen in Figures~\ref{figure_divide_surface} and~\ref{figure_ordered_morse}.

Divides were first introduced by A'Campo (\cite{acampo_sur_1973}\cite{acampo_groupe_1975}) and Gusein-Zade (\cite{gusein-zade_intersection_1974},\cite{gusein-zade_dynkin_1974}), who independently proved that they could be associated in a natural way to singularities and used them for studying properties of the monodromy. Later on, in \cite{acampo_real_1999},\cite{acampo_generic_1998} A'Campo associated to any divide $\mathcal{D}$ a link $L(\mathcal{D})$, constructed as follows. Consider the tangent bundle of the unit disk, $TD^2 = \{(x,v) \mid x\in D^2, v\in T_xD^2\}$. The sphere $S^3$ can be seen as the unit sphere in $TD^2$, 
$$S^3 = \{(x,v) \in TD^2 \mid |x|^2 + |v|^2 = 1\}.$$
Now let $\mathcal{D} \subset D^2$ be a divide, the link of $\mathcal{D}$ is defined as
$$ L(\mathcal{D}) = \{(x,v) \in S^3 \mid x\in \mathcal{D}, v\in T_x\mathcal{D}\} \subset S^3.$$

This gives a link whose number of components is equal to the number of intervals in the divide. In the same papers, A'Campo proved that if the divide is connected the link is fibred and that if the divide was obtained from a singularity the associated link $L(\mathcal{D})$ is ambient isotopic to the link of the singularity. In this latter case, he also provided an easy graphical algorithm to construct a model of the Milnor fibre on which a system of vanishing cycles is visible. We say that a \textit{face} of a divide $\mathcal{D}$ is a connected component of $D^2\setminus \mathcal{D}$ which does not intersect the boundary of $D^2$. Let $n$ be the number of intervals in $\mathcal{D}$, $\delta$ be the number of crossings and $r$ the number of faces. The Milnor fibre will be a surface with first Betti number $\mu = \delta + r$ and $n$ boundary components. The distinguished vanishing cycles will be given by one curve per crossing and one curve per face. The surface is constructed as follows: first, replace every crossing of $\mathcal{D}$ with a small circle, to get a trivalent graph. Now, realize every edge of this new graph by a half-twisted band. This will give a surface composed of twisted cylinders, corresponding to the crossings of $\mathcal{D}$, connected by half-twisted bands corresponding to the edges of $\mathcal{D}$. The vanishing cycle associated to a crossing will be given by the core curve of the corresponding cylinder, the vanishing cycle of a face will be given by the core curves of the bands bounding the face. An example of this construction is given in Figure~\ref{figure_divide_surface}. 

\begin{figure}
	\centering
	\def\svgwidth{.6\columnwidth}
	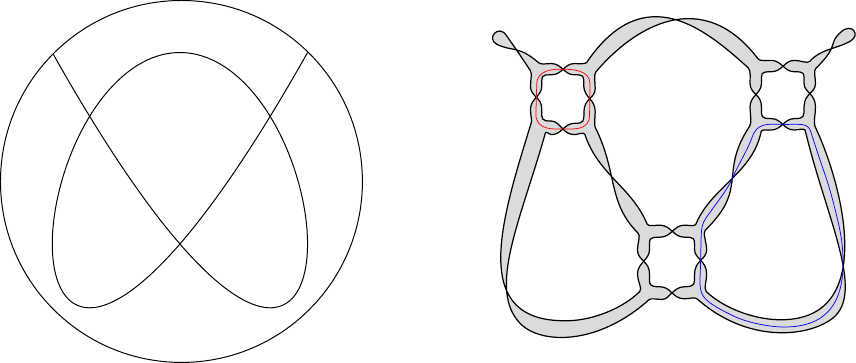
	\caption{A divide and the associated surface with some of the vanishing cycles. The corresponding link is the torus knot $T_{3,4}$.}
	\label{figure_divide_surface}
\end{figure}

\begin{rem}
A'Campo's construction only leads to a combinatorial model of the Milnor fibre which is not embedded. A graphical procedure to construct a diagram of the link of a divide and the associated embedded fibre surface has been given by Hirasawa in \cite{hirasawa_visualization_2002}.
\end{rem}

\begin{defn}
	Let $f$ be an isolated plane curve singularity, $\mathcal{D}$ a divide associated to $f$ and $\Sigma(f)$ the surface constructed from $\mathcal{D}$ with the previous procedure. The \textit{geometric monodromy group} of $f$ is the subgroup of $\mathrm{Mod}(\Sigma(f))$ generated by the Dehn twists around the vanishing cycles constructed on $\Sigma(f)$. This does not depend on the choice of the divide $\mathcal{D}$.
\end{defn}

As we have already mentioned, links of singularities are closures of positive braids. Since fibre surfaces of fibred links are unique, the Milnor fibre of a singularity $f$ is ambient isotopic to the fibre surface $\Sigma_{\beta}$ of any positive braid $\beta$ representing $L(f)$. We therefore now have two a priori distinct subgroups of $\mathrm{Mod}(\Sigma_{\beta}) = \mathrm{Mod}(\Sigma(f))$, the geometric monodromy group of $f$ and the monodromy group of $\beta$. Theorem~\ref{theorem_group_singularity} says that those two groups coincide for at least one choice of $\beta$.

To prove Theorem~\ref{theorem_group_singularity}, we will explicitly find an isotopy between the Milnor fibre constructed from a divide and the surface of an appropriate positive braid and identify the vanishing cycles on this braid surface. In order to do so, we need to use a divide from which the positive braid is somehow visible.

\begin{defn}
	A divide $\mathcal{D} \subset D^2$ is an \textit{ordered Morse divide} if there is a diameter of $D^2$ such that the orthogonal projection on this diameter is Morse when restricted to $\mathcal{D}$, all the local maxima (resp. minima) have the same critical value $b$ (resp. $a$) with $b>a$ and all the crossings are mapped in the open interval $(a,b)$.
\end{defn}

Basically, a divide is ordered Morse (w.r.t. a given direction) if no local maxima or minima lie in an interior face of the divide. Examples of such divides are given in Figure~\ref{figure_ordered_morse}.

\begin{figure}
	\centering
	\def\svgwidth{0.5\columnwidth}
	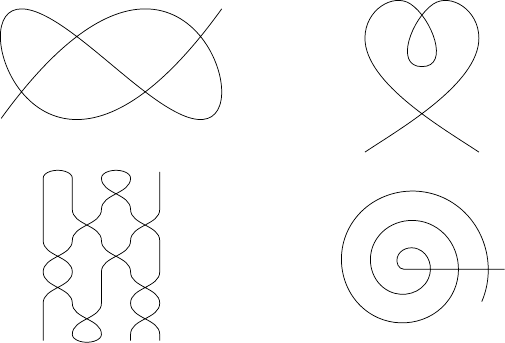
	\caption{The divides on the left are ordered Morse, the divides on the right are not.}
	\label{figure_ordered_morse}
\end{figure}

\begin{rem}
	In the literature, ordered Morse divides are sometimes called \textit{scannable divides}.
\end{rem}

Ordered Morse divides were introduced by Couture and Perron \cite{couture_representative_2000}, who used a generalization of those to construct a representative braid for any divide link. In particular, ordered Morse divides give positive braid links. Notice that every singularity has an associated divide which is ordered Morse (in fact, the divides originally constructed by A'Campo and Gusein-Zade are ordered Morse, see \cite{couture_representative_2000}). The result of Couture and Perron can be obtained geometrically: if we apply the algorithm of \cite{hirasawa_visualization_2002} to an ordered Morse divide, we get exactly the fibre surface of a positive braid. This was done in \cite{goda_lissajous_2002} for Lissajous divides and torus links, but the same procedure works for an arbitrary ordered Morse divide. The construction of the fibre surface is shown in Figure~\ref{figure_divide_embedded}: one just has to replace the crossings and minima/maxima of the divide with the corresponding pieces of surface and glue them together following the pattern of the divide. Here we use that all the minima and maxima of the divide are in the exterior face: for general divides the fibre surface is more complicated.

\begin{rem}
	The diagrams in Figure~\ref{figure_divide_embedded} are the mirror image of those obtained by Hirasawa in \cite{hirasawa_visualization_2002}. This is due to the different choice of orientation of $S^3$: Hirasawa uses the orientation induced by the trivialization $T\mathbb{R}^2 = \{(x,v) \mid x\in \mathbb{R}^2 , v \in T_x\mathbb{R}^2\} \cong \mathbb{R}^2 \times \mathbb{R}^2 $; we use the identification $T\mathbb{R}^2 \cong \mathbb{C}^2$, where the plane $\mathbb{R}^2$ is identified with the real part of $\mathbb{C}^2$, since this allows to correctly identify the link of a singularity with the link of a corresponding divide. 
\end{rem}

\begin{figure}
	\centering
	\def\svgwidth{.95\columnwidth}
	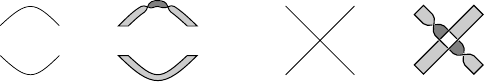
	\caption{Hirasawa's construction of the embedded fibre surface of an ordered Morse divide.}
	\label{figure_divide_embedded}
\end{figure}

\begin{proof}[Proof of Theorem~\ref{theorem_group_singularity}]
	Let $f$ be an isolated plane singularity and $\mathcal{D}$ an associated ordered Morse divide. Let $\Sigma$ be the embedded surface constructed following \cite{hirasawa_visualization_2002}, as explained above. It is an embedded fibre surface whose boundary is the link $L(\mathcal{D}) = L(f)$. To see that this is indeed the fibre surface of a positive braid, we just need to perform the isotopies shown in Figure~\ref{figure_isotopy_divide_braid} (2a), getting a collection of disks connected by half-twisted bands, and slide all the bands to the front. Let us remark that an ordered Morse divide is formed of $N$ parallel lines (where $N$ is the number of points in the preimage of a regular value of the Morse projection) connected by the crossings and the minima/maxima. The braid obtained will have $N$ strands, a crossing of $\mathcal{D}$ gives a pair of generators while every maximum/minimum gives one generator.
	
	By further performing the isotopies of Figure~\ref{figure_isotopy_divide_braid}, (2b) around all the crossings of $\mathcal{D}$ corresponding to generators $\sigma_i$ for even $i$, we can now directly identify $\Sigma$ with an embedded version of A'Campo's model of the Milnor fibre. A system of vanishing cycles is therefore visible on the braid surface $\Sigma$. Those cycles are not exactly the same as the generators of the monodromy group of the braid, but the same arguments as in the proof of Proposition~\ref{prop_elem_conj} show that the two groups are indeed the same.
\end{proof}

\begin{figure}
	\centering
	\def\svgwidth{.85\columnwidth}
	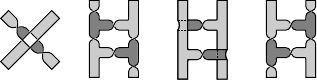
	\caption{A sequence of isotopies.}
	\label{figure_isotopy_divide_braid}
\end{figure}

\begin{exmp}	
In Figure~\ref{figure_isotopy_example}, we see an example of the isotopies used in the previous proof. On the left, we start with a divide $\mathcal{D}$; we then construct the Seifert surface following Hirasawa's algorithm. After applying the isotopies of Figure~\ref{figure_isotopy_divide_braid} (2a), we obtain the surface $\Sigma_{\beta}$ of a positive braid, namely $\beta=(\sigma_1\sigma_2\sigma_3)^3$. On the right, we performed the isotopy of Figure~\ref{figure_isotopy_divide_braid} (2b) around the central crossing of $\mathcal{D}$. In that way, we clearly see that the surface is composed of twisted cylinders corresponding to the crossings of $\mathcal{D}$ and connected by \textit{half-twisted} bands, as required by A'Campo's construction (compare with Figure~\ref{figure_divide_surface}). Notice that it is not relevant that this last step is performed around all the crossings of $\mathcal{D}$ corresponding to generators $\sigma_i$ for \textit{even} $i$ as opposed to \textit{odd} $i$; what matters is that it alternates, in order the get the required half-twisting of the bands become visible. 

\begin{figure}
	\centering
	\def\svgwidth{.95\columnwidth}
	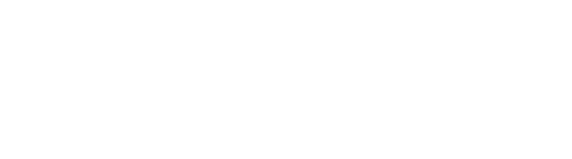
	\caption{An example of the isotopies of Theorem~\ref{theorem_group_singularity}.}
	\label{figure_isotopy_example}
\end{figure}

\end{exmp}

\section{Framings}\label{section_framings}

We will now briefly recall the basics of the theory of framed surfaces, concentrating in particular on the action of the mapping class group on such structures, as investigated in \cite{calderon_framed_2020} and \cite{randal-williams_homology_2014}. In what follows, we will adhere to the notations and conventions of \cite{calderon_framed_2020}, but we will restrict only to the case of surfaces with connected boundary. Let $\Sigma = \Sigma_{g,1}$ be a connected, compact, oriented surface of genus $g$ with one boundary component. A framing $\phi$ on $\Sigma$ is a trivialization of the tangent bundle $T\Sigma$. With the fixed orientation (and a choice of a Riemannian metric), a framing is determined by a nowhere-vanishing vector field $\xi_\phi$ on $\Sigma$. Two framings are \textit{isotopic} if the associated vector fields are isotopic through nowhere-vanishing vector fields.

To a framing one can associate a \textit{winding number function}, computing the holonomy of a simple closed curve. If $c: \mathbb{S}^1 \rightarrow \Sigma$ is a $\mathcal{C}^1$ embedding, one can define $$\phi(c) = \int _{\mathbb{S}^1} d \angle (\dot{c}(t), \xi_\phi(c(t))) \in \mathbb{Z}. $$
This defines a map from the set of simple closed curves on $\Sigma$ to $\mathbb{Z}$, which is clearly invariant under isotopy of $\phi$ and $c$. It is not hard to see that the converse also holds: the isotopy class of a framing on $\Sigma$ is determined by its winding number function, and actually by the value on finitely many curves (see \cite{calderon_framed_2020} Lemma 2.2 and \cite{randal-williams_homology_2014} Prop.2.4). Thanks to this, we will use the term "framing" indifferently to refer to the isotopy class of the vector field $\xi_\phi$ or to the associated winding number function $\phi$.

\begin{rem}
	Since we are only considering surfaces with connected boundary, it follows from the Poincaré-Hopf index theorem that for any framing $\phi$ on $\Sigma$, if the boundary $\partial\Sigma$ is oriented with the surface on its left, $\phi(\partial\Sigma) = \chi(\Sigma)$.
\end{rem}

The mapping class group of $\Sigma$ acts on the set of isotopy classes of framings by pullback, via $f \cdot \phi (c) = \phi (f^{-1}(c))$, for $f \in \mathrm{Mod}(\Sigma)$ and $c$ a simple closed curve.

\begin{defn}
	Let $(\Sigma, \phi)$ be a framed surface. The \textit{framed mapping class group} $$\mathrm{Mod}(\Sigma,\phi) = \{f\in \mathrm{Mod}(\Sigma) \mid f \cdot \phi = \phi\}$$ is the stabilizer of the isotopy class of $\phi$.
\end{defn}

Of particular interest is the action of Dehn twists.

\begin{lem}[\cite{calderon_framed_2020}, Lemma $2.4$]\label{lemma twist-linearity}
	Let $(\Sigma,\phi)$ be a framed surface and $a,x$ oriented simple closed curves on $\Sigma$, then 
	$$ \phi(T_a(x)) = \phi(x) + \langle x,a \rangle \phi(a),$$
	where $\langle \cdot, \cdot \rangle$ denotes the algebraic intersection number. 
\end{lem}

We say that a nonseparating simple closed curve $a$ on $(\Sigma,\phi)$ is \textit{admissible} if $\phi(a) = 0$. As a consequence of Lemma~\ref{lemma twist-linearity} we have that a nonseparating simple closed curve $a \subset \Sigma$ is admissible if and only if the corresponding Dehn twist preserves $\phi$. Calderon and Salter proved that, for big enough genus, the framed mapping class group is generated by those admissible twists:
\begin{prop}[\cite{calderon_framed_2020}, Prop. $5.11$]
	If $(\Sigma, \phi)$ is a framed surface of genus $g \geq 5$,
	$$ \mathrm{Mod}(\Sigma,\phi) = \langle T_a \mid a \text{ admissible for } \phi \rangle .$$
\end{prop}

But more is true. The framed mapping class group is generated by finitely many admissible twists around curves with prescribed intersection pattern. Again following \cite{calderon_framed_2020}:

\begin{defn}
	Let $\mathcal{C} = \{c_1, \cdots, c_k\}$ be a collection of curves on a surface $\Sigma$, pairwise in minimal position and intersecting at most once. We say that such a configuration:
	\begin{itemize}
		\item \textit{spans the surface} if $\Sigma$ deformation retracts onto the union of curves in $\mathcal{C}$;
		\item is \textit{arboreal} if its intersection graph is a tree, and \textit{$E$-arboreal} if moreover it contains the Dynkin diagram $E_6$ as a subtree.
	\end{itemize}
\end{defn}

\begin{defn}
	 Let $\mathcal{C} = \{c_1, \cdots , c_k,c_{k+1}, \cdots, c_l\}$ be a collection of curves on a surface $\Sigma$ and denote by $S_j$ a regular neighbourhood of $\{c_1, \cdots , c_j\}$.
	 We say that $\mathcal{C}$ is an \textit{$h$-assemblage of type $E$} if:
	\begin{itemize}
		\item $\{c_1,\cdots , c_k\}$ is an $E$-arboreal spanning configuration on a subsurface $S \subset \Sigma$ of genus $h$;
		
		\item For $j > k$, $c_j \cap S_{j-1}$ is a single arc;
		
		\item $S_l = \Sigma$.
	\end{itemize}
\end{defn}

\begin{prop}[\cite{calderon_framed_2020}, Theorem B] \label{prop_finite_generation}
	Let $(\Sigma, \phi)$ be a framed surface and $\mathcal{C} = \{c_1, \cdots ,c_l\}$ an $h$-assemblage of type $E$ on $\Sigma$ of genus $h \geq 5$. If all the curves in $\mathcal{C}$ are admissible for $\phi$, then 
	$$ \mathrm{Mod}(\Sigma,\phi) = \langle T_c \mid c\in \mathcal{C} \rangle .$$
	
\end{prop}

The orbit space of this action was studied by Randal-Williams in \cite{randal-williams_homology_2014}. It is classified by the Arf invariant. More precisely, it follows from work of Johnson \cite{johnson_spin_1980} that the function $(\phi + 1)$ mod $2$ is a quadratic refinement of the mod $2$ intersection form. We can therefore define $\mathcal{A}(\phi)$ to be the Arf invariant of this quadratic form. More concretely, let us denote by $\mathit{i}(\cdot,\cdot)$ the geometric intersection number and take a collection of oriented simple closed curves $\{x_1,y_1, \dots,x_g,y_g\}$ such that $\langle x_i, x_j \rangle = \langle y_i, y_j \rangle = 0$ and $\langle x_i,y_j \rangle = \mathit{i}(x_i,y_j) = \delta_{i,j}$. We then have
$$ \mathcal{A}(\phi) = \sum\limits_{i=1}^g (\phi(x_i)+1)(\phi(y_i)+1) \mod 2.$$ This is of course independent of the choice of the curves $\{x_1,y_1, \dots,x_g,y_g\}$.

\begin{prop}[\cite{randal-williams_homology_2014}, Theorem $2.9$]
	Let $g\geq 2$. The action of the mapping class group on the set of isotopy classes of framings on $\Sigma = \Sigma_{g,1}$ has exactly two orbits, distinguished by the Arf invariant.
\end{prop}

As a consequence, for a given $\Sigma$ there are exactly two conjugacy classes of framed mapping class groups as subgroups of $\mathrm{Mod}(\Sigma)$.

\begin{rem}[Caveat] In this section we only stated results for surfaces with connected boundary, in terms of \textit{absolute} framings. For general surfaces, the whole theory is still valid, but needs to be formulated for \textit{relative} framings, i.e. only allowing isotopies that are trivial on the boundary. In this more general context, the framed mapping class group is the stabilizer of the \textit{relative} isotopy class of a framing, and one needs to also take into account the action on arcs, getting so-called generalized winding number functions. The orbit space is now classified by a generalized Arf invariant together with the values of the framing on the different boundary components. However, if the boundary is connected the absolute and relative theories are equivalent and we can use this slightly simpler formulation.
	
\end{rem}

\subsection{A framing for positive braids}

Let $\beta$ be a non-split positive braid and $\Sigma_{\beta}$ its fibre surface. We can construct a framing $\phi_{\beta}$ on $\Sigma_{\beta}$ as in Figure~\ref{framing_example}. An explicit and straightforward computation now shows that every standard curve $\gamma_i$ on $\Sigma_{\beta}$ is admissible for $\phi_{\beta}$. Therefore, the monodromy group of $\beta$ is contained in the framed mapping class group of $\phi_{\beta}$:
$$ \mathit{MG}(\beta) \leqslant \mathrm{Mod}(\Sigma_{\beta},\phi_{\beta}) .$$

\begin{figure}
	\centering
	\def\svgscale{0.4}
	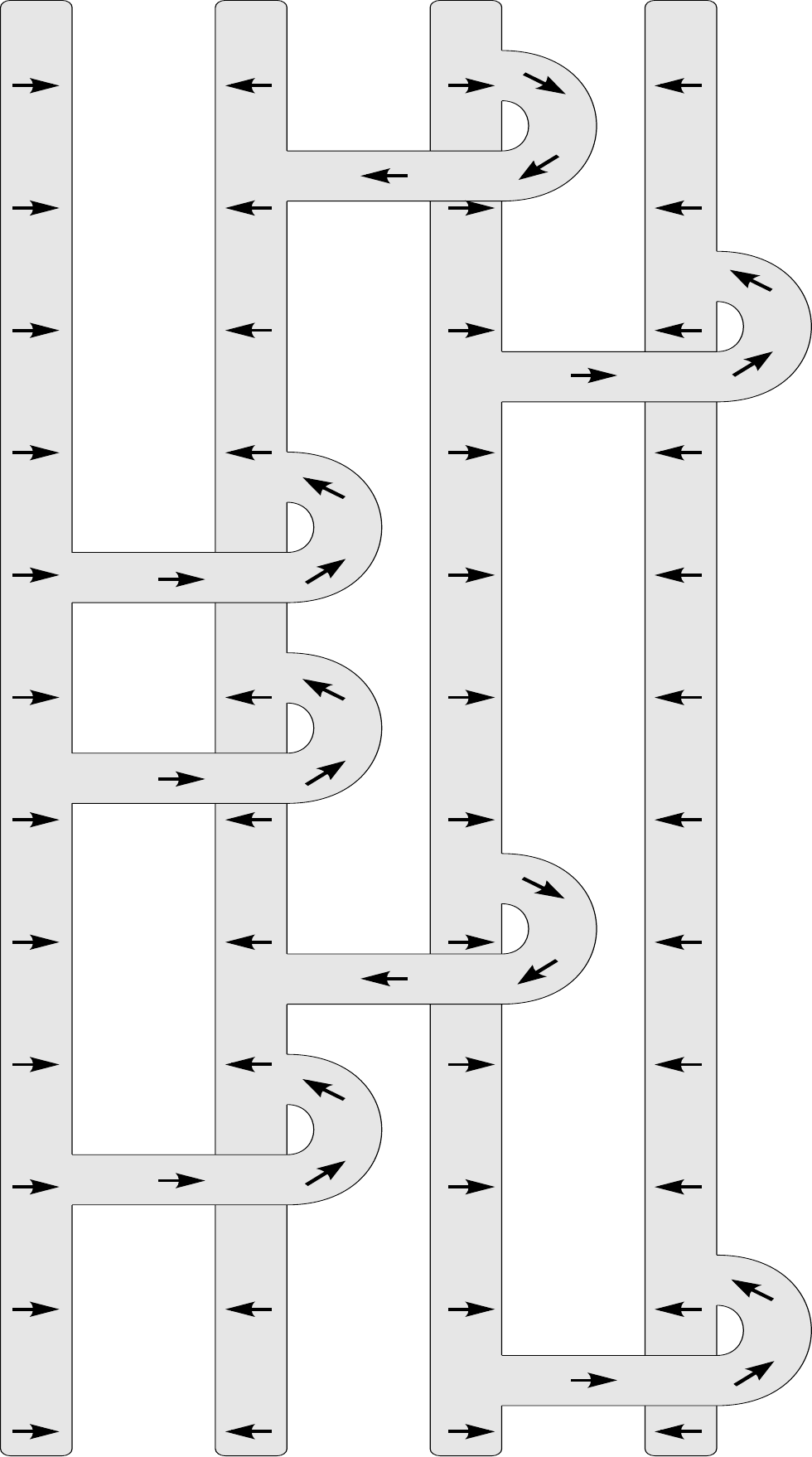
	\caption{The framing on $\Sigma_{\beta}$ for $\beta = \sigma_3\sigma_1\sigma_2\sigma_1^2\sigma_3\sigma_2$. On the vertical disks it is horizontal with alternating directions, on the twisted bands it is parallel to the core.}
	\label{framing_example}
\end{figure}

We will prove that, at least for positive braids whose closure is a knot of big enough genus, the monodromy group is equal to this framed mapping class group. Therefore, in view of the previous discussion, we now want to compute the Arf invariant of $\phi_{\beta}$.

\begin{prop}\label{prop_arf_invariant}
	Let $\beta$ be a positive braid whose closure is a knot $K$. Then
	$$\mathcal{A}(\phi_{\beta}) = \mathcal{A}(K),$$ where $\mathcal{A}(K)$ is the classical Arf invariant of $K$.
\end{prop}

To prove Proposition~\ref{prop_arf_invariant}, we will need to discuss a bit more in detail the Arf invariant. Let $V$ be a finite dimensional vector space over $\mathbb{Z}_2$ equipped with a non-singular, symmetric bilinear pairing $\langle \cdot, \cdot \rangle : V \times V \rightarrow \mathbb{Z}_2$. Recall that a \textit{quadratic refinement} of the bilinear pairing $\langle \cdot, \cdot \rangle$ is a function $q: V \rightarrow \mathbb{Z}_2$ such that for all $x,y \in V$
$$ q(x+y) = q(x) + q(y) + \langle x,y \rangle.$$ To such a mod $2$ quadratic form it is classically associated the Arf invariant $\mathcal{A}(q) \in \mathbb{Z}_2$.

In our context, we will take $V = H_1(\Sigma_{\beta},\mathbb{Z}_2)$ and $\langle \cdot, \cdot \rangle$ the mod $2$ intersection form. As we have already mentioned, the framing $\phi_{\beta}$ induces a quadratic refinement of the intersection form, whose Arf invariant is $\mathcal{A}(\phi_{\beta})$. On the other hand, if the closure of $\beta$ is a knot $K$, it is known that the Seifert form also induces such a quadratic refinement. More precisely, if $\mathit{S}: H_1(\Sigma_{\beta}) \times H_1(\Sigma_{\beta}) \rightarrow \mathbb{Z}$ denotes the Seifert form, we can define $q: V \rightarrow \mathbb{Z}_2$ by $q(x) = \mathit{S}(x,x) \mod 2$. It is a classical result that the Arf invariant of this quadratic form is indeed an invariant of $K$, that we denote by $\mathcal{A}(K)$.

\begin{proof}[Proof of Proposition~\ref{prop_arf_invariant}]
	Let $\beta$ be a positive braid whose closure is a knot $K$ and $\Sigma_{\beta}$ its fibre surface, equipped with the framing $\phi_{\beta}$. The family of curves $\gamma_i$ form a basis of $V = H_1(\Sigma_{\beta},\mathbb{Z}_2)$. Since by construction all the $\gamma_i$ are admissible for $\phi_{\beta}$, for every $i$ we have the equality
	$$\phi_{\beta}(\gamma_i) + 1 = 1 = \mathit{S}(\gamma_i,\gamma_i) \mod 2.$$
	Since $\{\gamma_i\}$ is a basis, it now follows from the defining equation of a quadratic refinement that for every $x \in V$
	$$\phi_{\beta}(x) + 1  = q(x) \mod 2.$$ Therefore the two quadratic forms $(\phi_{\beta} + 1)$ mod $2$ and $q$ coincide, so their Arf invariants also do. 
\end{proof}

\section{Proof of the main theorem}\label{section_proof}

In this section we will give the proof of Theorem~\ref{theorem_framed}, stating that, up to finitely many exceptions, the monodromy group of a positive braid not of type $A_n$ and whose closure is a knot is a framed mapping class group. In the previous section we have constructed a framing $\phi_{\beta}$ on the fibre surface $\Sigma_{\beta}$ and seen that $ \mathit{MG}(\beta) \leqslant \mathrm{Mod}(\Sigma_{\beta},\phi_{\beta})$, so we only need to deal with the opposite inclusion. This will be done by applying Proposition~\ref{prop_finite_generation}. As a first step, we have to find appropriate subsurfaces supporting an $E$-arboreal spanning configuration. For this, we will separately consider the case of braids on $3$-strands (Prop.~\ref{prop 3-braids}), on at least $11$ strands (Prop.~\ref{prop big braid index}) and finally with an intermediate number of strands (Prop.~\ref{prop_inter_braid_index}).

\begin{prop}\label{prop 3-braids}
	Let $\beta$ be a prime positive $3$-braid of genus $g\geq 5$ which is not of type $A_n$ or $D_n$. Then, excepted finitely many braids, up to positive braid isotopy its linking graph contains an induced subtree which is an $E$-arboreal spanning configuration on a subsurface of genus $g \geq 5$.
\end{prop}

\begin{proof}
	Let $\beta$ be a positive $3$-braid which is not of type $A_n$. Up to elementary conjugation and braid relation we can assume that
	$\beta = \sigma_1^{a_1}\sigma_2^{b_1} \cdots \sigma_1^{a_m}\sigma_2^{b_m},$
	with $a_i \geq 2$ and $b_i \geq 1$ for all $i \in \{1,\cdots,m\}$. First of all, notice that if we can find a suitable subtree for a braid $\sigma_1^{a_1}\sigma_2^{b_1} \cdots \sigma_1^{a_m}\sigma_2^{b_m}$, the result will also hold for any braid $\sigma_1^{a'_1}\sigma_2^{b'_1} \cdots \sigma_1^{a'_m}\sigma_2^{b'_m}$ for $a'_i \geq a_i$ and $b'_i\geq b_i$.	We will now prove the result by case distinction over $m$.
	
	\begin{itemize}
		\item $m\geq 5$ : Every braid with $m\geq 5$ has genus $g\geq 5$ so it is clearly enough to prove the result for $m = 5$. If one of the $b_i$ is at least $2$, we can assume that $\beta = \sigma_1^2\sigma_2\sigma_1^2\sigma_2\sigma_1^2\sigma_2\sigma_1^2\sigma_2\sigma_1^2\sigma_2^2$. In the left of Figure~\ref{mgeq5} we now see an induced subtree of the linking graph with the required properties. Similarly if one of the $a_i$ is at least $3$ we can assume that $\beta = \sigma_1^2\sigma_2\sigma_1^2\sigma_2\sigma_1^2\sigma_2\sigma_1^3\sigma_2\sigma_1^2\sigma_2$, and we find the induced subtree of the right of Figure~\ref{mgeq5}.
		
		\begin{figure}
			\centering
			\begin{minipage}{0.4\textwidth}
				\centering
				\includegraphics{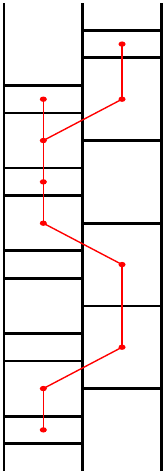}
			\end{minipage}%
			\begin{minipage}{0.4\textwidth}
				\centering
				\includegraphics{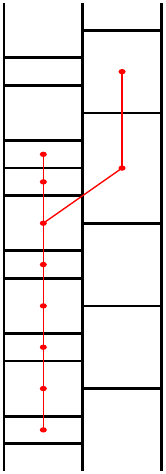}
			\end{minipage}
			\caption{Subtrees of $\sigma_1^2\sigma_2\sigma_1^2\sigma_2\sigma_1^2\sigma_2\sigma_1^2\sigma_2\sigma_1^2\sigma_2^2$ and $\sigma_1^2\sigma_2\sigma_1^2\sigma_2\sigma_1^2\sigma_2\sigma_1^3\sigma_2\sigma_1^2\sigma_2$}
			\label{mgeq5}
		\end{figure}
		
		We are now only left with the braid $\sigma_1^2\sigma_2\sigma_1^2\sigma_2\sigma_1^2\sigma_2\sigma_1^2\sigma_2\sigma_1^2\sigma_2$. Here we do not directly find an appropriate subtree, but Figure~\ref{m=5} shows a sequence of braid relations that makes it visible.
		
		\begin{figure}
			\centering
			\def\svgwidth{0.4\columnwidth}
			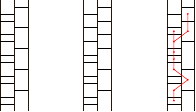
			\caption{The braid $\sigma_1^2\sigma_2\sigma_1^2\sigma_2\sigma_1^2\sigma_2\sigma_1^2\sigma_2\sigma_1^2\sigma_2$}
			\label{m=5}
		\end{figure}

		\item $m=4$ : We will treat several cases. Let us first assume that there is an $i$ such that $b_i\geq 2$. If there are $i\neq j$ such that $b_i,b_j \geq 2$, then up to cyclic ordering we only have to deal with the two cases depicted in the left of Figure~\ref{m=4_first}, where we see the sought subtrees. Similarly, if there is only one $b_i$ greater than $2$ but there is one $a_j$ bigger than $3$ we will find one of the trees in the right of Figure~\ref{m=4_first}. Finally, if all the $a_j$ are equal to $2$ and there is only one $b_i$ greater than $2$, it is enough to consider the braid $\sigma_1^2\sigma_2\sigma_1^2\sigma_2\sigma_1^2\sigma_2\sigma_1^2\sigma_2^2$, for which we can find the subtree after applying some braid relations as in Figure~\ref{m=4_second}.
		
		\begin{figure}
			\centering
			\def\svgwidth{0.45\columnwidth}
			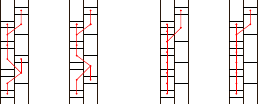
			\caption{The cases when $m=4$ and $b_i,b_j \geq 2$ (left) or $b_i\geq 2$ and $a_j \geq 3$ (right)}
			\label{m=4_first}
		\end{figure}
	
		\begin{figure}
			\centering
			\def\svgwidth{.55\columnwidth}
			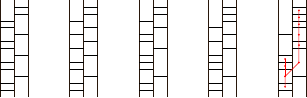
			\caption{The braid $\sigma_1^2\sigma_2\sigma_1^2\sigma_2\sigma_1^2\sigma_2\sigma_1^2\sigma_2^2$}
			\label{m=4_second}
		\end{figure}
		
		We are now left with $b_i = 1$ for all $i$. Notice that in that case there need to be at least one $a_i \geq 3$, otherwise the braid has genus less than $5$. If there are two non-consecutive $a_i$ and $a_j$ greater than $3$, it is enough to consider the braid $\sigma_1^3\sigma_2\sigma_1^2\sigma_2\sigma_1^3\sigma_2\sigma_1^2\sigma_2$, for which we find an appropriate subtree in the left of Figure~\ref{m=4_b=1}. If not, up to cyclic ordering there must be two consecutive $a_i=a_{i+1}=2$, in which case we can apply a sequence of braid relations as we did in the right of Figure~\ref{m=4_b=1} and find our subtree.
		
		\begin{figure}
			\centering
			\def\svgwidth{.7\columnwidth}
			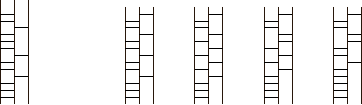
			\caption{The cases $m=4$ and $b_i=1$ for all $i$}
			\label{m=4_b=1}
		\end{figure}
	
	\item $m=3$ : This will be the lengthier case, since there are many low genus braids that require special treatment. Let $\beta = \sigma_1^{a_1}\sigma_2^{b_1}\sigma_1^{a_2}\sigma_2^{b_2}\sigma_1^{a_3}\sigma_2^{b_3}$ be a braid of genus $g\geq 5$, then a simple argument implies that $\sum a_i + \sum b_i \geq 12$.
	If $\sum b_i \geq 8$, it is enough to consider the braids in Figure~\ref{m=3_b8}. Similarly, when $\sum a_i \geq 11$ it is enough to consider the case when all the $b_i$ are equal to $1$, and up to elementary conjugation we can assume that $a_3 \geq 3$. In this case, by taking all the vertices in the left column and only the topmost of the right column we will always end up finding a tripod tree $T(1,k,9-k)$ for $k\geq 2$, which all correspond to subsurfaces of genus $5$, see Figure~\ref{m=3_a11} for some examples.
			
	\begin{figure}
		\centering
		\def\svgwidth{.8\columnwidth}
		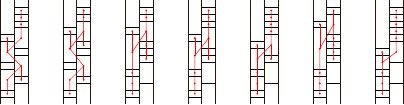
		\caption{When $m=3$ and $\sum b_i = 8$}
		\label{m=3_b8}
	\end{figure}

	\begin{figure}
	\centering
	\def\svgwidth{.65\columnwidth}
	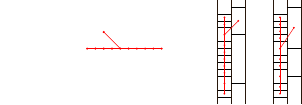
	\caption{The tripod trees for $m=3$ and $\sum a_i = 11$}
	\label{m=3_a11}
	\end{figure}

	We are now left with the low genus cases.
	
	\begin{itemize}
		\item $\sum a_i = 6$ : If $\sum b_i = 6$ we always get a link with $3$ components and genus $4$. If $\sum b_i = 7$ and there is at least one of the $b_i$ equal to one, up to elementary conjugation we can assume that $\beta = \sigma_1^2\sigma_2\sigma_1^2\sigma_2^{b_2}\sigma_1^2\sigma_2^{b_3}$ with $b_2+b_3 = 6$. Using that $\sigma_1^2 \sigma_2 \sigma_1^2 $ commutes with $\sigma_2$ we get $\sigma_2^{b_2}\sigma_1^2\sigma_2\sigma_1^4\sigma_2^{b_3}$, which is conjugate to $\sigma_1^2\sigma_2\sigma_1^4\sigma_2^6$, whose intersection graph is a tree with the required properties. 
		We are now left with $b_i=2$ for all $i$. Up to elementary conjugation there is only one such braid, $\sigma_1^2\sigma_2^2\sigma_1^2\sigma_2^2\sigma_1^2\sigma_2^3$. Here there are no possible braid relations to apply and it is not possible to find a subtree of big enough genus.

		\item $\sum a_i = 7$ : Let us first assume that $\sum b_i = 6$. If there is at least one $b_i$ equal to one, we can directly find our subtrees. In Figure~\ref{m=3_a7_b6} we see some of the cases. The omitted ones are symmetric and will give the same subtrees. Notice that this will also cover all the braids with $\sum a_i \geq 7$ and $\sum b_i \geq 7$. If $b_i = 2$ for all $i$, up to elementary conjugation there is only the braid $\sigma_1^3\sigma_2^2\sigma_1^2\sigma_2^2\sigma_1^2\sigma_2^2$, for which again we cannot find any subtree of big enough genus.

		\begin{figure}
			\centering
			\def\svgwidth{.75\columnwidth}
			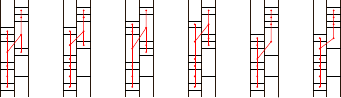
			\caption{When $\sum a_i = 7$ and $\sum b_i = 6$, with $b_1 = 1$}
			\label{m=3_a7_b6}
		\end{figure}
		
		If $\sum b_i = 5$, up to conjugation we have $\beta =\sigma_1^3\sigma_2^{b_1}\sigma_1^2\sigma_2^{b_2}\sigma_1^2\sigma_2^{b_3}$. If $b_2 = 1$, using that $\sigma_1^2 \sigma_2 \sigma_1^2 $ commutes with $\sigma_2$ we get the braid $\sigma_1^5 \sigma_2\sigma_1^2 \sigma_2^4$, whose intersection graph is a tree with the required properties. We are left with the three braids $\sigma_1^3\sigma_2\sigma_1^2\sigma_2^2\sigma_1^2\sigma_2^2$, $\sigma_1^3\sigma_2^2\sigma_1^2\sigma_2^2\sigma_1^2\sigma_2$ and $\sigma_1^3\sigma_2\sigma_1^2\sigma_2^3\sigma_1^2\sigma_2$. For the first, up to elementary conjugation and applying the commutativity relation as before we have 
		$$\sigma_1^3\sigma_2\sigma_1^2\sigma_2^2\sigma_1^2\sigma_2^2 \simeq \sigma_1^2\sigma_2^2\sigma_1^3\sigma_2\sigma_1^2\sigma_2^2 = \sigma_1^2\sigma_2^2\sigma_1\sigma_2^2\sigma_1^2\sigma_2\sigma_1^2 = \sigma_1^4\sigma_2^2\sigma_1\sigma_2^3\sigma_1^2 \simeq \sigma_1^6\sigma_2^2\sigma_1\sigma_2^3$$ and we get a suitable tree. The second braid is symmetric and will lead to the same intersection tree. For the last, we similarly get
		$$\sigma_1^3\sigma_2\sigma_1^2\sigma_2^3\sigma_1^2\sigma_2 = \sigma_1\sigma_2^3\sigma_1^2\sigma_2\sigma_1^4\sigma_2 \simeq \sigma_2\sigma_1\sigma_2^3\sigma_1^2\sigma_2\sigma_1^4 = \sigma_1^3\sigma_2\sigma_1^3\sigma_2\sigma_1^4 \simeq \sigma_1^7\sigma_2\sigma_1^3\sigma_2.$$
		
		\item $\sum a_i = 8$ : If $\sum b_i \geq 6$, then either we are already done by the case $\sum a_i = 7$ (if one of the $b_i$ is equal to one) or it is symmetric to the case $\sum b_i \geq 8$. If $\sum b_i = 5$ after applying some positive braid isotopy we can always find an appropriate subtree, with the lone exception of $\beta = \sigma_1^3\sigma_2\sigma_1^3\sigma_2^2\sigma_1^2\sigma_2^2$, for which we couldn't find any. In Figure~\ref{m=3_a8_b5} we see some of the cases, the remaining ones being braid isotopic to those. Finally, if $\sum b_i = 4$, we only get links of $3$ components and genus $4$ excepted for the braid $\beta = \sigma_1^3\sigma_2\sigma_1^3\sigma_2\sigma_1^2\sigma_2^2$ (and the symmetric $\beta = \sigma_1^2\sigma_2\sigma_1^3\sigma_2\sigma_1^3\sigma_2^2$), for which we see the tree in Figure~\ref{m=3_a8_b4}.
		
		\begin{figure}[h]
			\centering
			\def\svgwidth{.8\columnwidth}
			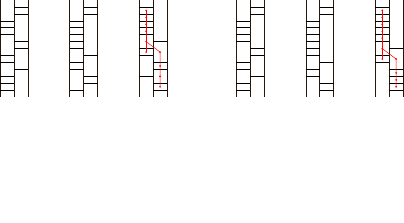
			\caption{When $\sum a_i = 8$ and $\sum b_i = 5$}
			\label{m=3_a8_b5}
		\end{figure}
		
		\begin{figure}
			\centering
			\def\svgwidth{.55\columnwidth}
			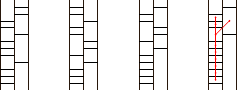
			\caption{The braid $\sigma_1^3\sigma_2\sigma_1^3\sigma_2\sigma_1^2\sigma_2^2$}
			\label{m=3_a8_b4}
		\end{figure}
	
		\item $\sum a_i = 9$ : If $\sum b_i \geq 4$, it is enough to consider $\beta = \sigma_1^{a_1}\sigma_2\sigma_1^{a_2}\sigma_2\sigma_1^{a_3}\sigma_2^2$. By taking all the vertices of the linking graph excepted the lowermost of the right column, according to the value of $a_3$ we will get one of the tripod trees $T(1,2,6)$, $T(2,2,5)$ and $T(3,2,4)$, which all correspond to surfaces of genus $5$. If $\sum b_i = 3$ and there is one even $a_i$, we only have to consider the three braids $\sigma_1^3\sigma_2\sigma_1^2\sigma_2\sigma_1^4\sigma_2$, $\sigma_1^3\sigma_2\sigma_1^4\sigma_2\sigma_1^2\sigma_2$ and $\sigma_1^5\sigma_2\sigma_1^2\sigma_2\sigma_1^2\sigma_2$. The first two are symmetric, and using that $\sigma_1^2\sigma_2\sigma_1^2$ commutes with $\sigma_2$ we see that the first one is braid equivalent to the last, for which we furthermore have $\sigma_1^5\sigma_2\sigma_1^2\sigma_2\sigma_1^2\sigma_2 = \sigma_1^7\sigma_2\sigma_1^2\sigma_2^2$, whose intersection graph is a tree. Finally, if all the $a_i$ are odd, we get a link of genus $4$.
		
		\item $\sum a_i = 10$ : The only case left is when $\sum b_i = 3$. If one of the $a_i$ is odd we can suppose that $a_3$ is odd, in which case by taking all the bricks excepted the lowermost of the right column we will get a tripod tree $T(1,2,6)$ or $T(1,4,4)$, which both correspond to subsurfaces of genus $5$. If all the $a_i$ are even, up to elementary conjugation we only have the braids $\sigma_1^4\sigma_2\sigma_1^4\sigma_2\sigma_1^2\sigma_2$ and $\sigma_1^6\sigma_2\sigma_1^2\sigma_2\sigma_1^2\sigma_2$. Those are actually related by braid relations and elementary conjugations, and the very same argument used for $\sum a_i = 9$ and $\sum b_i = 3$ will yield the required tree.
		
	\end{itemize}

	\item $m=2$ : For a braid $\beta = \sigma_1^{a_1}\sigma_2^{b_1}\sigma_1^{a_2}\sigma_2^{b_2}$ of genus at least $5$ the intersection graph is always a tree with at least $10$ crossings. Furthermore, by direct inspection we see that those trees will always contain $E_6$ unless they are of type $D_n$.
	
	\item $m=1$ : In this case we only get non-prime braids.

	\end{itemize}

	To sum up, the result holds for all braids excepted $\sigma_1^2\sigma_2^2\sigma_1^2\sigma_2^2\sigma_1^2\sigma_2^3$, its symmetric $\sigma_1^3\sigma_2^2\sigma_1^2\sigma_2^2\sigma_1^2\sigma_2^2$ (which gives the same link with opposite orientation) and $\sigma_1^3\sigma_2\sigma_1^3\sigma_2^2\sigma_1^2\sigma_2^2$ (which gives an invertible link). 
\end{proof}

\bigskip

We will now consider braids with big positive braid index.

\begin{prop}\label{prop big braid index}
	Let $\beta$ be a prime positive braid on $N \geq 11$ strands and whose closure is a knot not of type $A_n$. Then, up to positive braid isotopy and excepted finitely many braids, its linking graph contains an induced subtree which is an $E$-arboreal spanning configuration on a subsurface of genus $g\geq 5$.
\end{prop}

 The strategy to prove Proposition~\ref{prop big braid index} is very simple: we will try to explicitly construct the required subtree and see that, each time our construction fails, either the closure is not a knot or we can reduce the number of strands. The finitely many exceptions come from Prop.~\ref{prop 3-braids} and Prop.~\ref{prop_inter_braid_index}, in case we can reduce our braid to one of the exceptions therein. We will therefore heavily rely on the following two lemmas.

\begin{lem}\label{lemma_braid_index}
	Let $\beta \in B_N^+$ be a prime positive braid on $N\geq 3$ strands. If for some $i$ the linking graph of the subword induced by all the generators $\sigma_i$ and $\sigma_{i+1}$ is a path, then there exists a positive braid $\beta' \in B_{N-1}^+$ such that $\hat{\beta} = \hat{\beta'}$ and $\mathit{MG}(\beta) = \mathit{MG}(\beta')$.
\end{lem}

\begin{proof}
	Up to elementary conjugation and symmetry, we can assume that the subword induced by $\sigma_i$ and $\sigma_{i+1}$ is of the form $\sigma_i^a\sigma_{i+1}\sigma_i\sigma_{i+1}^b$. Moreover, we can suppose that all the generators $\sigma_j$ for $j<i$ appear before the last occurrence of $\sigma_i$ and all the generators $\sigma_j$ for $j>i+1$ appear after the first occurrence of $\sigma_{i+1}$. In Figure~\ref{figure_braid_index_lemma} we see an isotopy between the fibre surface $\Sigma_{\beta}$ and the fibre surface $\Sigma_{\beta'}$ of a new braid $\beta'$ with one strand less: the portion of the $(i+1)$-th disk lying between the first occurrence of $\sigma_{i+1}$ and the last occurrence of $\sigma_i$ (in red in the leftmost picture) is slid along the last $\sigma_i$, becoming a band between the $i$-th and $(i+1)$-th disk (central image); this band is then slid along the back of the two disks to be brought in the lowermost position.  A direct computation now shows that $\mathit{MG}(\beta) = \mathit{MG}(\beta')$.
	
	\begin{figure}
		\centering
		\def\svgwidth{\columnwidth}
		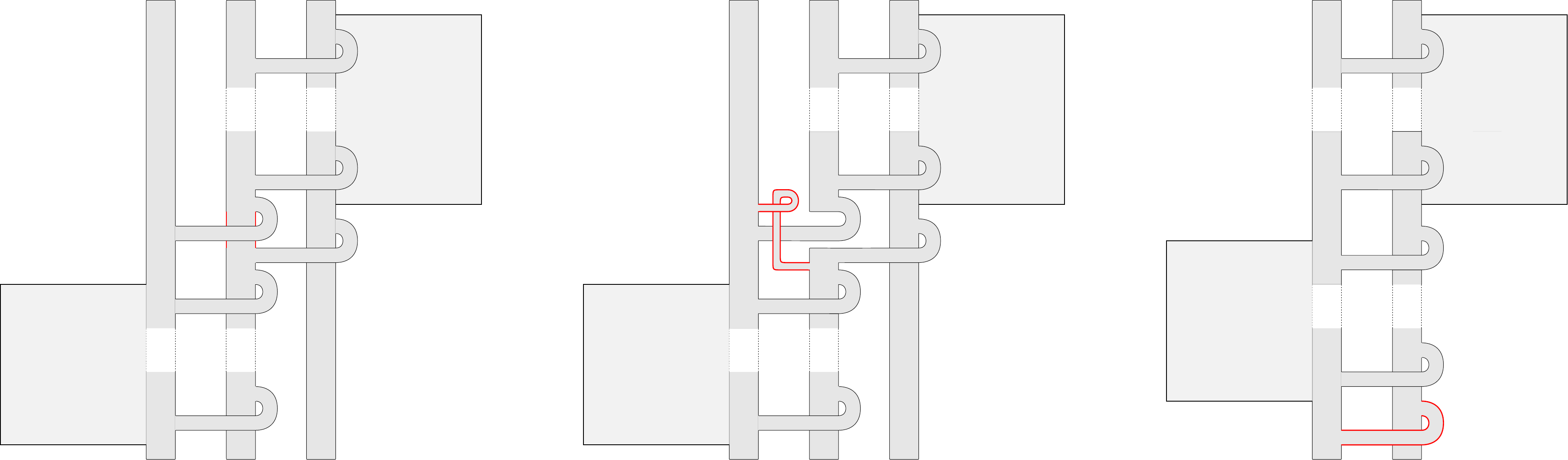
		\caption{An isotopy that reduces the number of strands.}
		\label{figure_braid_index_lemma}
	\end{figure}

\end{proof}

\begin{lem}\label{lemma_links}
 Let
 \begin{itemize}
 	\item  $A = \{ \sigma_1^a\sigma_2\sigma_3^b\sigma_2\sigma_1^c\sigma_2\sigma_3^d\sigma_2\sigma_1^e \mid a,b,c,d,e \in \mathbb{N} \}$,
 	
 	\item $B = \{ \beta_1\sigma_2\sigma_3\beta_2\sigma_3\sigma_2\beta_3 \mid  \beta_1,\beta_3 \in \langle \sigma_3, \sigma_4 \rangle, \beta_2 \in \langle \sigma_1,\sigma_2 \rangle \}$,
 	
 	\item  $C = \{ \beta_1\sigma_2\beta_2\sigma_2\sigma_3\beta_3\sigma_3\beta_4 \mid \beta_1,\beta_4 \in \langle \sigma_1,\sigma_4 \rangle, \beta_2 \in \langle\sigma_3,\sigma_4\rangle, \beta_3\in\langle\sigma_1\sigma_2\rangle \}$.
 	
\end{itemize}	

If $\beta \in A \cup B \cup C$, then the closure of $\beta$ has at least two components.
\end{lem}

\begin{proof}
	In Figure~\ref{figure_lemma_links} we see some schematic drawings of the linking diagrams of braids from the three families, in which one component of the closure is highlighted.
	
	\begin{figure}
		\centering
		\def\svgwidth{.7\columnwidth}
		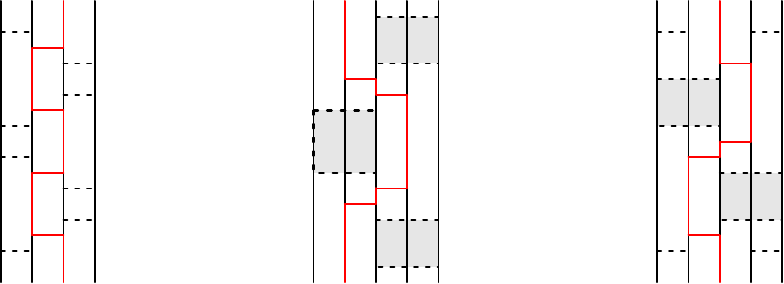
		\caption{Some positive braids with disconnected closure.}
		\label{figure_lemma_links}
	\end{figure}

\end{proof}

Notice that, even though for sake of simplicity we only stated Lemma~\ref{lemma_links} for braids with few strands, the result clearly also applies in case some columns of the brick diagram of a braid on more strands exactly look as in Figure~\ref{figure_lemma_links} (or are symmetric to those).

To construct the trees required in Proposition~\ref{prop big braid index}, we will also need the following lemma from \cite{liechti_genus_2020}.

\begin{lem}[\cite{liechti_genus_2020}, Lemma $7$]\label{lemma_path}
	Let $\beta $ be a prime positive braid and $v$ be a vertex of its linking graph. Then there is an induced path in the linking graph connecting $v$ to any other column of the brick diagram.
\end{lem}

We will briefly recall the algorithm for constructing such a path, since this will be used in what follows. Let us say that we want to connect $v$ to a column to its right. Start at $v$ and move up or down its column until reaching the closest brick linked to the right (potentially, already $v$). Now, move to the right and repeat the procedure. If at the moment of moving to the right there are several possibilities, choose the brick which is the closest to a brick in the same column linked again to its right. It is easy to see that those choices prevent the creation of cycles, so that the result will be a path.

\begin{proof}[Proof of Proposition \ref{prop big braid index}]
	Let $\beta$ be a prime positive braid on $N\geq 11$ strands. By Lemma \ref{lemma_braid_index} we can assume that, for every pair of adjacent columns in the brick diagram, the linking graph restricted to those columns is not a path. Let us furthermore repeatedly apply all the possible braid relations of the form $\sigma_i \sigma_{i+1} \sigma_i \rightsquigarrow \sigma_{i+1}\sigma_i\sigma_{i+1}$, until no subword $\sigma_i \sigma_{i+1} \sigma_i$ is left in $\beta$. Our strategy goes as follows: we will start considering an induced path connecting the leftmost column to the rightmost, constructed with the previous algorithm, and try to add to it one single vertex, in order to get a tripod tree containing $E_6$. Since $b\geq 11$, the tripod tree will have at least $11$ vertices and hence correspond to a subsurface of genus at least $5$. So, let us fix one such path and look at the third column of the brick diagram. If we can add a brick of this column to the path and get an (induced) tripod tree we are done. There are two reasons why this might not be possible: either because there are no leftover bricks in the third column or because every available brick is linked to more than one brick of the path and adding it would generate a cycle. We will now analyse those cases in detail.
	By symmetry, we can assume that in the third column our path arrives from the left to a brick $v_3$, potentially moves \textit{down} to a brick $w_3$ and then continues to the right.
	
	\bigskip 
	
	If there are no leftover bricks in the third column, then by the construction rule of our paths we know that $w_3$ is the only brick of column $3$ linked to the right. We can now apply elementary conjugations on the right-hand side of the diagram in order to have all the generators $\sigma_i$ for $i\geq 4$ appear before the last occurrence of $\sigma_3$, and perform again all the possible braid moves $\sigma_i \sigma_{i+1} \sigma_i \rightsquigarrow \sigma_{i+1}\sigma_i\sigma_{i+1}$. Those transformations will not affect the first $3$ columns and the part of the path therein. We now get that the sub-braid generated by $\sigma_3$ and $\sigma_4$ is $\sigma_4^a\sigma_3\sigma_4^b\sigma_3^c$, with $c \geq 1$ and $a,b \geq 2$ by Lemma~\ref{lemma_braid_index}. Let us denote by $v_4$ the only brick of column $4$ linked to $w_3$, and let us attach a path connecting $v_4$ to the rightmost column.
		
	If at least one of the bricks immediately above or below $v_4$ is not linked to the portion of the path in the fifth column (in particular, if $v_4$ is itself linked to the right), it can safely be added to get a tripod tree. We directly see that we are left with the case of Figure~\ref{Nobricks}. Notice that, up to modifying the path in the fourth and fifth columns, we can always choose whether it passes by $x$ or $y$. Now, if there is a brick $x'$ above $x$, either it is not linked to the path in the fifth column, in what case we can directly connect it to $x$, or it is, in what case we can change our path to $w_3\rightarrow v_4\rightarrow x \rightarrow x'\rightarrow \{\textit{path in the fifth column}\}$ (thus avoiding $v_5$) and connect $y$ to $v_4$. Similarly, we can assume that there is no brick below $y$.
		
	Let us now consider the fifth column. Notice that there must be at least one brick immediately above and one immediately below $v_5$ that are not linked to the fourth column, otherwise we could apply one of the forbidden braid relations. By applying the same reasoning as before, we conclude that we can always obtain a tripod tree, unless there are no other bricks in the column. In the latter case, however, the closure of the braid is not a knot by Lemma~\ref{lemma_links} (compare with the leftmost diagram of Figure~\ref{figure_lemma_links}).
		
	\begin{figure}[h]
		\centering
		\def\svgwidth{.16\columnwidth}			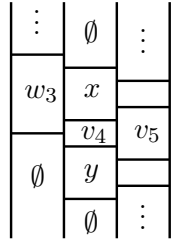
		\caption{We only show the columns $3-5$. The path goes through $w_3$, $v_4$, either $x$ or $y$ and $v_5$.}		\label{Nobricks}
	\end{figure}
	
	\bigskip
	
	We can now suppose that there are some leftover bricks in the third column, but adding any of them to our path creates a cycle. The idea is analogous to what we just did: we will try to locally "reconstruct" the linking graph, successively exclude all the cases where we can find the required tripod and see that in the end we are left with one of the links from Lemma~\ref{lemma_links}. However, the analysis gets much more delicate and will need lengthy case distinctions to cover the various ways adjacent columns can be connected. First of all, in the third column there could be bricks left both above and below the path, only above or only below.
	
	\begin{enumerate}
		\item\label{case_ab-bel} If there are bricks above $v_3$ and below $w_3$, we will be in one of the two cases of Figure~\ref{Ab-Bel}. 
		
		\begin{figure}[h]
			\centering
			\def\svgwidth{.5\columnwidth}			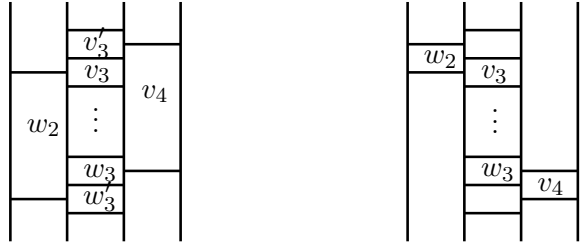
			\caption{In both cases the path arrives from $w_2$, moves to $v_3$, then goes down to $w_3$ and finally to $v_4$.}
			\label{Ab-Bel}
		\end{figure}
	
		\begin{enumerate}
			\item\label{case_ab-bel_A} In the left-hand case of Figure~\ref{Ab-Bel}, recalling that the path was constructed with the algorithm of Lemma~\ref{lemma_path}, we know that either $v_3$ and $w_3$ are adjacent or they coincide. We will analyse those cases in great detail, since they serve as example of the kind of reasoning applied also to the rest of the proof.
			
			\begin{enumerate}
				\item\label{case_ab-bel_A_a} If $v_3$ and $w_3$ are distinct and adjacent, as in the left of Figure~\ref{Ab-Bel_Aa}, again by the construction rule of our paths we know that $w_3'$ is not linked to the right at all and $v_3'$ is not linked to the path to the left. Now, if $v_3'$ is linked to the path to the right above $v_4$, we could change our path to $w_2 \rightarrow v_3 \rightarrow v_3' \rightarrow \{\textit{path in the fourth column}\}$, thus avoiding $v_4$, and connect $w_3$ to $v_3$ to get a tripod (see center of Figure~\ref{Ab-Bel_Aa}). Otherwise, we can instead consider $w_2 \rightarrow w_3' \rightarrow w_3 \rightarrow v_4 \rightarrow \{\textit{path}\}$ and connect $v_3'$ to $v_4$ (right of Figure~\ref{Ab-Bel_Aa}).
				
				\begin{figure}[h]
					\centering
					\def\svgwidth{\columnwidth}			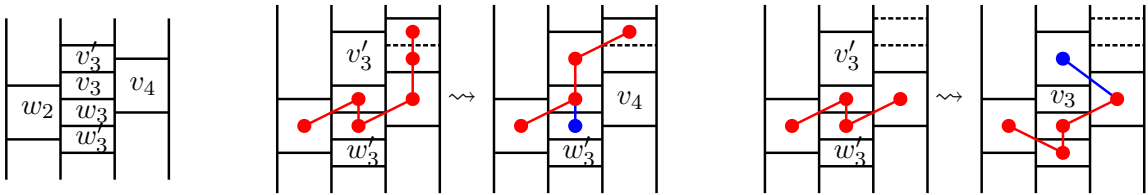
					\caption{Diagrams for case~\ref{case_ab-bel_A_a}; in each case, we drew on the left the original path, on the right the modified path with in blue the isolated vertex of the tripod.}
					\label{Ab-Bel_Aa}
				\end{figure}
				
				\item\label{case_ab-bel_A_b} If $v_3 = w_3$, then we know that $w_2$ has to be linked to the first column, otherwise we could perform one of the forbidden braid relations. We will further distinguish according to how $w_2$ is linked to the first column.
				
				\begin{enumerate}
							
				\item\label{case_ab-bel_A_b1} Let us suppose first that $w_2$ is linked to a brick $v_1$ below it, as in the left-hand side of Figure~\ref{Ab-Bel_Ab1}. Notice that the brick denoted by $w_2'$ needs to exist because of the condition on the possible braid relations. Hence, we can assume that in the first column there are at most two bricks, both linked with $w_2$, and that the brick immediately below $w_2'$ (if any) is linked with $v_1$, otherwise we could immediately find an appropriate tripod, as shown in Figure~\ref{Ab-Bel_Ab1}.
				
				\begin{figure}[h]
					\centering
					\def\svgwidth{\columnwidth}			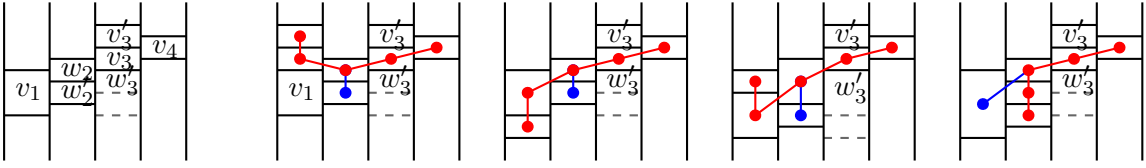
					\caption{First diagrams for case~\ref{case_ab-bel_A_b1}. In the third column, the brick $w_3'$ is linked to $w_2$ and may or may not be linked to $w_2'$.}
					\label{Ab-Bel_Ab1}
				\end{figure}

				We are therefore left we the diagram on the left-hand side of Figure~\ref{Ab-Bel_Ab1_2}. If there is a linking between the second and third columns above $v_3$, we could modify our path by starting from $v_1$ and $w_2$, then moving upwards in the second column until we reach the first connection with the third column above $v_3$ and finally going down on the third column until the first connection to the original path in the fourth column (which occurs at the latest at $v_3'$). This will give us a path avoiding $v_3$. We can now safely connect $w_3'$ to $w_2$ and get a tripod. If not, up to elementary conjugations on the first two columns, we can suppose that there are no bricks in the second column above $v_3$, as in the central picture of Figure~\ref{Ab-Bel_Ab1_2}. In this case, we can assume that above $w_2$ there is at most one brick. Now, if in the first column there are two bricks, again by elementary conjugation we are back to the case where there is a brick below $v_1$ and we are done. We are hence left with just one brick in the first column, as in the right-hand side of Figure~\ref{Ab-Bel_Ab1_2}. Notice that in this case the brick $w_3'$ is forced to be linked to $w_2'$, otherwise the closure of the braid is not a knot by the second case of Lemma~\ref{lemma_links}. This in turn forces the existence of the brick denoted by $b$ below $w_3'$, otherwise we could apply a forbidden braid relation. If there is a brick $a$ below $w_2'$, we can consider $v_1 \rightarrow a \rightarrow w_2' \rightarrow w_3' \rightarrow v_3 \rightarrow \{\textit{path}\} $ and connect $b$ to $w_3'$. On the other hand, if there are no bricks below $w_2'$ we see that either the closure of the braid is not a knot, if there is a brick above $w_2$ (third case of Lemma~\ref{lemma_links}), or we can reduce the number of strands with Lemma~\ref{lemma_braid_index}.

				\begin{figure}[h]
					\centering
					\def\svgwidth{.8\columnwidth}			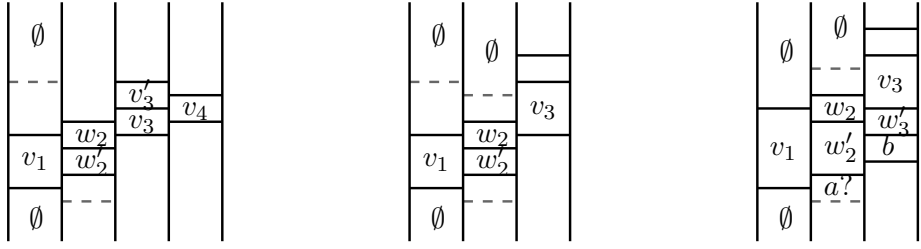
					\caption{Further diagrams for case~\ref{case_ab-bel_A_b1}. The dashed lines show where the following brick (if existing) would be. In the third column, there is still a brick $w_3'$ below $v_3$ as in Figure~\ref{Ab-Bel}, which is linked to $w_2$ and may or may not be linked to $w_2'$.}
					\label{Ab-Bel_Ab1_2}
				\end{figure}

				\item\label{case_ab-bel_A_b2} We can now suppose that $w_2$ is linked to a brick $v_1$ above it, but is not linked with any brick of the first column below it, as in the leftmost image of Figure~\ref{Ab-Bel_Ab2}. If there are at least two bricks below $w_2$ we immediately find a tripod. If there is exactly one brick below $w_2$, we can furthermore assume that $v_1$ is the only brick in the first column. Let us now consider how $v_1$ is connected with the second column. If it is only linked to $w_2$, by applying an elementary conjugation we are back case~\ref{case_ab-bel_A_b1}, where $v_1$ was below $w_2$. Notice that the existence of a brick below $w_2$ ensures that the condition about the possible braid relations is still satisfied after the conjugation. If $v_1$ is linked to another brick of the second column above $w_2$, called $a$, and $a$ is below $v_3'$, as in the second image of Figure~\ref{Ab-Bel_Ab2}, we immediately see that either we find a suitable tripod or the closure is not a knot, depending on how many bricks there are in the second column between $w_2$ and $a$ (there is at least one by the condition on braid relations; if it is unique, we fall in the second case of Lemma~\ref{lemma_links}, else we find a tripod). Finally, if $a$ is above $v_3'$ or linked to it, as in the two right-hand side images of Figure~\ref{Ab-Bel_Ab2}, we know that there is a brick between $a$ and $w_2$ linked to $v_3$ (potentially, this could be $a$). We can now consider $v_1 \rightarrow a \rightarrow \{\textit{second column}\} \rightarrow v_3 \rightarrow \{\textit{path}\}$, thus avoiding $w_2$, and connect $w_3'$ to $v_3$.
				
				\begin{figure}[h]
					\centering
					\def\svgwidth{\columnwidth}			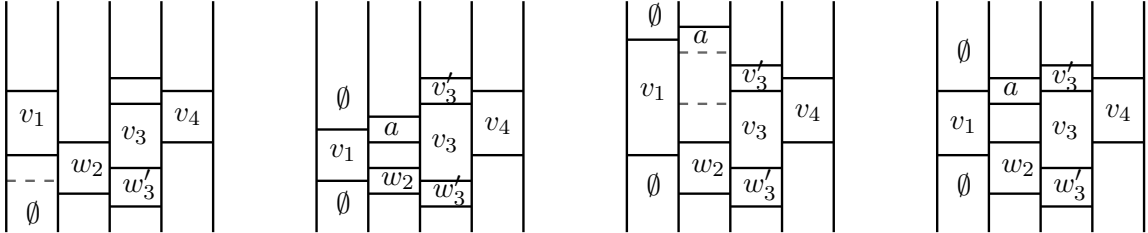
					\caption{Diagrams for case \ref{case_ab-bel_A_b2}. In the second column, there is at least one brick above $w_2$ but below $v_1$.}
					\label{Ab-Bel_Ab2}
				\end{figure}
			
				The only case left now is when there are no bricks below $w_2$. Again, if $v_1$ is linked to another brick $a$ of the second column above $w_2$ the exact same argument as before applies. If $v_1$ is linked only to $w_2$, this time we cannot simply apply an elementary conjugation to reduce to a previously treated case. However, if there are no bricks above $v_1$ (resp. below $v_1$) we could apply Lemma~\ref{lemma_braid_index}, whilst if there are bricks in the first column both above and below $v_1$ it is immediate to conclude that either we find a tripod or the closure is not a knot, as in the first case of Lemma~\ref{lemma_links}.

				\end{enumerate}
			
			\end{enumerate}
		
			\item\label{case_ab-bel_B} In the right-hand case of Figure~\ref{Ab-Bel}, we know that $w_2$ needs to be linked to a brick $v_1$ in the first column. Again, we will separately consider whether $v_1$ is above or below $w_2$.
			
			\begin{enumerate}
				\item\label{case_ab-bel_B_a} Suppose first that $w_2$ is linked to a brick $v_1$ above it, as depicted in the left of Figure~\ref{Ab-Bel_Ba1}. Note that the brick denoted by $v_2$ must exist, otherwise we could perform a forbidden braid relation. By excluding all the cases where one can immediately find a tripod, we are left with at most two bricks in the first column, both linked to $w_2$, and we know that the brick above $v_2$ (if any) is linked to $v_1$, see Figure~\ref{Ab-Bel_Ba1}.
				
				\begin{figure}[h]
					\centering
					\def\svgwidth{.9\columnwidth}			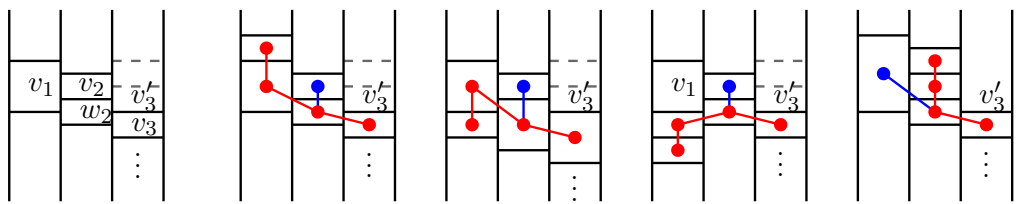
					\caption{First diagrams for case~\ref{case_ab-bel_B_a}. The dashed lines show where the brick $v_3'$ could end.}
					\label{Ab-Bel_Ba1}
				\end{figure}
			
				We hence can reduce the study to one of the cases in the left-hand side of Figure~\ref{Ab-Bel_Ba}. If there are two bricks in the first column, we either have a brick above $v_2$, in which case we can find a tripod by simply starting our path from $v_1'$ and adding two bricks above $w_2$, or we can apply an elementary conjugation to the first column to get a brick below $v_1'$, which again immediately gives a tripod. If in the first column there is just one brick, we know that $v_2$ needs to be linked to the third column, otherwise the closure is not a knot by Lemma~\ref{lemma_links} (second case). Thus, we can now suppose that there are no bricks above $v_2$, otherwise we immediately find a tripod, so we are left with the diagram on the right-hand side of Figure~\ref{Ab-Bel_Ba}. Notice that now by Lemma~\ref{lemma_braid_index} there needs to be at least one brick below $w_2$, otherwise we can reduce the number of strands. If none of the bricks below $w_2$ is linked to $v_3$, we see that according to the number of those bricks we either get a tripod or the closure is not a knot by (a symmetry of) the third case in Lemma~\ref{lemma_links}. Hence we can suppose that there is a brick $w_2'$ below $w_2$ linked to $v_3$. If $w_2'$ is connected to the original path in the third column below $v_3$, we can instead consider $v_1 \rightarrow w_2 \rightarrow \cdots \rightarrow w_2' \rightarrow \{\textit{path}\}$ and get a tripod by connecting to $w_2$ the bricks $v_3'$ and $v_3''$. If not, we can simply take our original path starting from $v_3$ and connect to it $w_2'$, $v_3'$ and $v_3''$.
		
				\begin{figure}[h]
					\centering
					\def\svgwidth{.8\columnwidth}			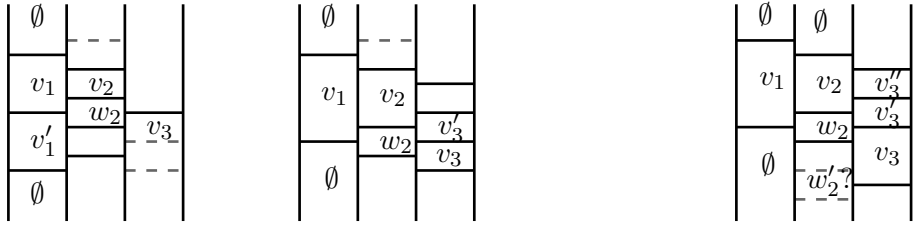
					\caption{Additional diagrams for case~\ref{case_ab-bel_B_a}.}
					\label{Ab-Bel_Ba}
				\end{figure}
	
				\item\label{case_ab-bel_B_b} Suppose now that $w_2$ is only linked to a brick $v_1$ below it, as in the leftmost image of Figure~\ref{Ab-Bel_Bb1}; note that, as depicted, there must exist one brick immediately below $w_2$ not linked to $v_1$, otherwise we could perform a braid relation. First, we immediately see that there can be at most one brick above $w_2$, and if this brick exists then $v_1$ is the only brick of the first column, otherwise we easily find a tripod. After excluding the additional easy cases shown in Figure~\ref{Ab-Bel_Bb1}, we are left with the diagrams of Figure~\ref{Ab-Bel_Bb}: that is, there must exist a brick $w_2'$ below $w_2$ which is linked to $v_3$ but above $v_1$.

				\begin{figure}[h!]
					\centering
					\def\svgwidth{\columnwidth}			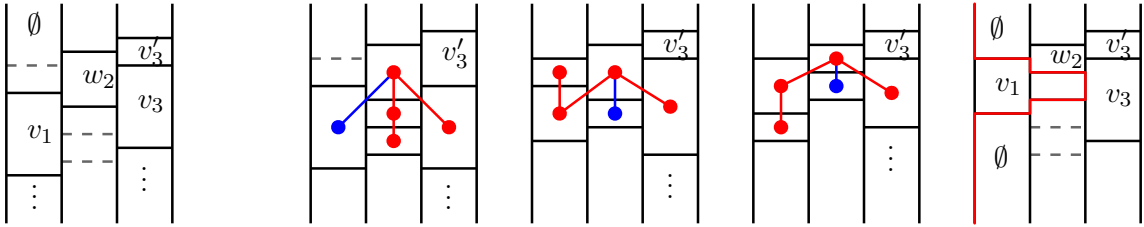
					\caption{First diagrams for case~\ref{case_ab-bel_B_b}.}
					\label{Ab-Bel_Bb1}
				\end{figure}

				First, if $w_2'$ is linked to the path in the third column below $v_3$, we can take $v_1 \rightarrow w_2 \rightarrow \cdots \rightarrow w_2' \rightarrow \textit{path}$ and add to it a brick in the third column (which will be at most $v_3''$). Otherwise, if $w_2'$ is not connected to the path and there is a brick $w_2''$ below it, we can simply take our original path from $v_3$ and add to it $v_3'$, $w_2'$ and $w_2''$. Finally, let's assume that there are no bricks below $w_2'$. If there is a brick above $w_2$ we can apply an elementary conjugation to the first column and get back to the previous case. If not, Lemma~\ref{lemma_braid_index} forces the existence of bricks above and below $v_1$, in which case either we get a tripod or the closure is not a knot, as in the first case of Lemma~\ref{lemma_links}.
			
				\begin{figure}[h]
					\centering
					\def\svgwidth{.45\columnwidth}			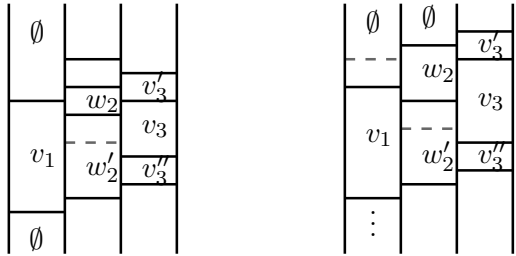
					\caption{Additional diagrams for case~\ref{case_ab-bel_B_b}.}
					\label{Ab-Bel_Bb}
				\end{figure}
			
			\end{enumerate}		
				
		\end{enumerate}
	
		\item\label{case_above} Let us now consider the case where there is at least one free brick $v_3'$ above $v_3$, but none below $w_3$. First of all, if after $w_3$ our path moves to a brick $v_4$ of the fourth column which is below it, we are basically in the same situation as Case~\ref{case_ab-bel_B}, and the precise same arguments apply. We can hence suppose that the path moves upwards in the fourth column. We will now treat different cases according to how $v_3'$ is linked to the neighbouring columns.
	
		\begin{enumerate}
			\item\label{case_above_A} If $v_3'$ is not linked to the right, we know that it needs to be linked to a brick $w_2$ in the second column, which in turns needs to be linked to a brick $v_1$ in the first column.
			\begin{enumerate}
				\item\label{case_above_A_a}	Let us suppose first that $v_1$ is above $w_2$, as in the leftmost image of Figure~\ref{Ab_A}. Note that we are in a situation similar to Case~\ref{case_ab-bel_B_a}, with the only difference that now the bricks above $v_3'$ could potentially be linked to the path in the fourth column; in particular, all the arguments therein still apply to the current situation, as long as they do not involve the bricks above $v_3'$. Hence, by Case~\ref{case_ab-bel_B_a}, we can suppose that there is only one brick in the first column, as in the central image of Figure~\ref{Ab_A}. Furthermore, if the brick $v_3''$ is not linked to its right, all the arguments from Case~\ref{case_ab-bel_B_a} still apply. We are then left with the rightmost diagram of Figure~\ref{Ab_A}. Now, if $v_3''$ is not linked to the path above $v_4$ it can directly be added as additional vertex, otherwise we can instead consider the path $v_1 \rightarrow w_2 \rightarrow v_3' \rightarrow v_3'' \rightarrow \{\textit{path}\}$ and add a brick to this new path in the fourth column.
				
				\begin{figure}[h]
					\centering
					\def\svgwidth{.8\columnwidth}			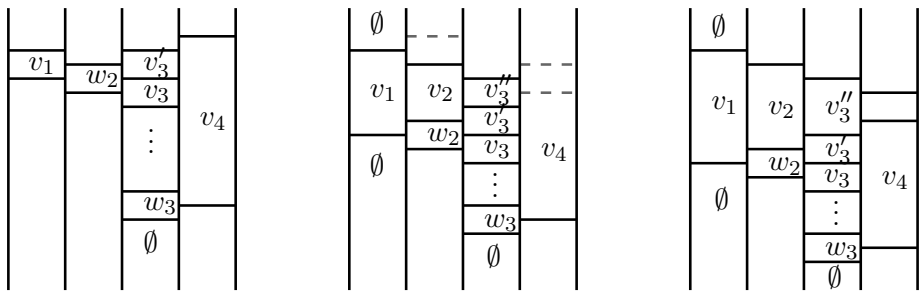
					\caption{Diagrams for Case~\ref{case_above_A}.}
					\label{Ab_A}
				\end{figure}
			
				\item\label{case_above_A_b}	If $v_1$ is below $w_2$, we are in a situation analogous to Case~\ref{case_ab-bel_B_b}, and in fact all the arguments therein still apply to the current setting, as we never made use of the bricks of the third column above $v_3'$. 
			\end{enumerate}

			\item\label{case_above_B} If $v_3'$ is linked to the right (to $v_4$) and to the left (to a brick $w_2$), by the construction rules of the path we know that either $v_3$ and and $w_3$ are adjacent or they coincide, and by the assumption on the braid relations $w_2$ is linked to a brick $v_1$ in the first column.
			
			\begin{enumerate}
				\item\label{case_above_B_a} If $v_1$ is above $w_2$, after repeating the arguments of Case~\ref{case_ab-bel_B_a} we can suppose that there is only one brick in the first column, so we are left with the two diagrams of Figure~\ref{Ab_Ba}.

				\begin{enumerate}
					\item\label{case_above_B_a1} Let us first consider the case where $v_3$ and $w_3$ are distinct and adjacent, as on the left of Figure~\ref{Ab_Ba}. If $v_3'$ is not linked to the path above $v_4$, we can simply consider $v_1 \rightarrow w_2 \rightarrow v_3' \rightarrow v_4 \rightarrow \{\textit{path}\}$ and add $v_3''$ (notice that this would also work if $v_3$ and $w_3$ did coincide). If $v_3'$ is linked to the path in the fourth column above $v_4$, take instead $v_2 \rightarrow v_3' \rightarrow \{\textit{path}\}$ and add $v_3$ and $w_3$. 
					
					\begin{figure}[h]
						\centering
						\def\svgwidth{.6\columnwidth}			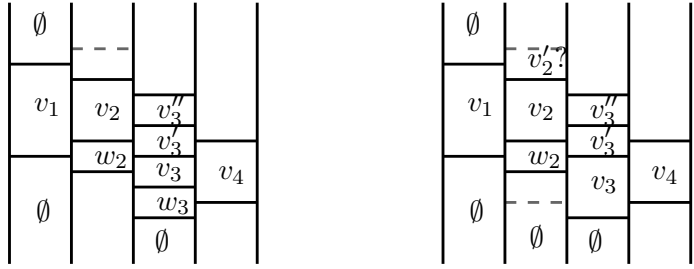
						\caption{Diagrams for Case~\ref{case_above_B_a}.}
						\label{Ab_Ba}
					\end{figure}
					
					\item Suppose now that $v_3$ and $w_3$ coincide, as on the right of Figure~\ref{Ab_Ba}. In this case, notice that no brick below $w_2$ can be linked to $v_3$ (otherwise we could perform a forbidden braid relation), and that therefore if there are at least two bricks below $w_2$ we immediately get a tripod. It follows that there needs to be a brick $v_2'$ above $v_2$, otherwise either we can apply Lemma~\ref{lemma_braid_index} (if there are no bricks below $w_2$) or the closure is not a knot, as in the third case of Lemma~\ref{lemma_links} (if there is exactly one brick below $w_2$). Now, if $v_3'$ is not linked to the path in the fourth column above $v_4$, we can find the same tripod as in Case~\ref{case_above_B_a1}. If $v_3'$ is linked to the path above $v_4$, we can instead consider $v_2' \rightarrow v_2 \rightarrow v_3' \rightarrow \{\textit{path}\}$ and add $v_3$.
					
				\end{enumerate}

				\item\label{case_above_B_b} Finally, if $v_1$ is below $w_2$, after repeating the arguments of Case~\ref{case_ab-bel_B_b} we are left with one of the diagrams of Figure~\ref{Ab_Bb}. Note that the case where $v_3$ and $w_3$ coincide is excluded by the condition on the braid relations. Furthermore, again by what was done in Case~\ref{case_ab-bel_B_b}, we know that we can assume the existence of a brick $w_2''$ below $w_2'$. Hence, if $v_3'$ is not connected to the path above $v_4$ we can take $v_1 \rightarrow w_2 \rightarrow v_3' \rightarrow v_4 \rightarrow \{\textit{path}\}$ and add $w_3$, if $v_3'$ is connected to the path above $v_4$ we can instead take $w_2'' \rightarrow w_2' \rightarrow v_3 \rightarrow v_3' \rightarrow \{\textit{path}\}$ and add $w_3$.
				
				\begin{figure}[h]
					\centering
					\def\svgwidth{.6\columnwidth}			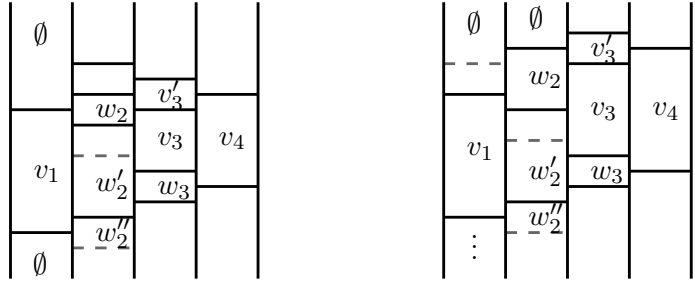
					\caption{Diagrams for Case~\ref{case_above_B_b}.}
					\label{Ab_Bb}
				\end{figure}
			
			\end{enumerate}
			
			\item\label{case_above_C} If $v_3'$ is not linked to the left, then it must be linked to the right to $v_4$. It follows that either $v_3$ and $w_3$ are adjacent or they coincide, as in Figure~\ref{Ab_C1}. In both cases, if $v_3'$ is connected to the path above $v_4$, we can simply let our path pass by $v_3'$ instead of $w_3$ (thus skipping $v_4$) and add a brick in the fourth column (which will be at most $v_4'$).
			
			\begin{figure}[h]
				\centering
				\def\svgwidth{.4\columnwidth}			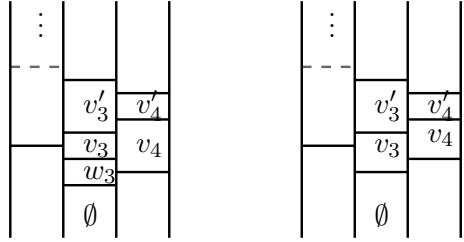
				\caption{Diagrams for Case~\ref{case_above_C}; $v_3$ is linked to the second column, but $v_3'$ is not.}
				\label{Ab_C1}
			\end{figure}	
			
			Suppose now that $v_3'$ is not connected to the path above $v_4$ and $v_3$,$w_3$ are distinct. If $w_3$ is linked to the left we are in the situation at the left-hand side of Figure~\ref{Ab_C2} and we directly find a tripod by considering $v_1 \rightarrow w_2 \rightarrow w_3 \rightarrow v_4 \rightarrow \{\textit{path}\}$ and adding $v_3'$. If not, we are in the situation at the right-hand side of  Figure~\ref{Ab_C2}. Note that this is analogous to Figure~\ref{Nobricks}, and the same arguments discussed there apply to the current setting.
			
			\begin{figure}[h]
				\centering
				\def\svgwidth{.45\columnwidth}			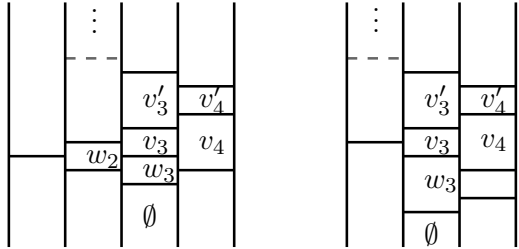
				\caption{Diagrams for Case~\ref{case_above_C}. On the left, we know that $w_2$ needs to be linked to some brick $v_1$ in the first column.}
				\label{Ab_C2}
			\end{figure}

			We are left with the case where $v_3$ and $w_3$ coincide and $v_3'$ is not connected to the path above $v_4$. We will now consider how the third and second column are connected.
			
			\begin{enumerate}
				\item\label{case_above_C_a} Let us suppose first that there is a brick $v_2$ in the second column below $v_3$. We know that $v_2$ needs to be linked to a brick in the first column, otherwise we could perform a forbidden braid relation.
				
				\begin{enumerate}
					\item\label{case_above_C_a1} If there is a brick $v_1$ in the first column above $v_2$, we are in one of the situations in the left of Figure~\ref{Ab_Ca1}. In both cases, we can assume that $v_1$ is the only brick of the first column linked to $v_2$, otherwise we find a tripod after elementary conjugation, as shown in the right of the figure. Moreover, in the leftmost case we now directly see that either we find a tripod (if there is at least another brick in the first column) or the closure is not a knot by Lemma~\ref{lemma_links}.
					
					\begin{figure}[h]
						\centering
						\def\svgwidth{\columnwidth}			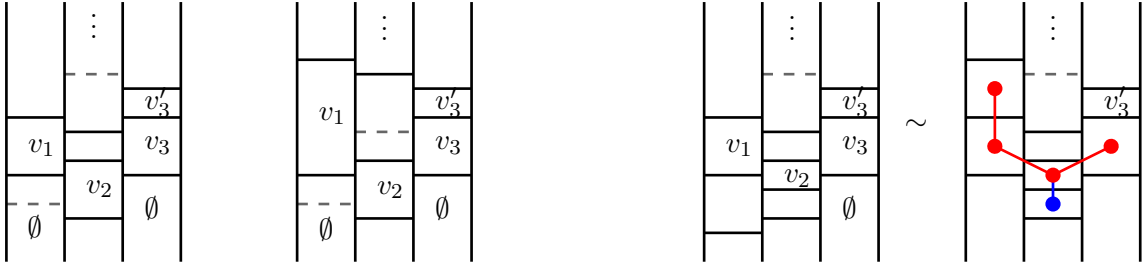
						\caption{Diagrams for Case~\ref{case_above_C_a1}, when $v_3$ and $w_3$ coincide and there is a brick $v_2$ below $v_3$.}
						\label{Ab_Ca1}
					\end{figure}
				
					Let us now focus on the second image of Figure~\ref{Ab_Ca1}. First of all, using Lemma~\ref{lemma_braid_index} we deduce that there must be a brick in the second column above $v_1$, as shown in the left of Figure~\ref{Ab_Ca1_2}. Note that we only drew the "extremal" cases; in all the others (having either more bricks below $v_2$ or more bricks in the first column), one can easily find a tripod. By excluding additional direct cases, we end up with the diagram on the right-hand side of Figure~\ref{Ab_Ca1_2}: indeed, we can assume that there is no brick below $v_2$, otherwise by elementary conjugations we would get two bricks above $v_2''$ and would find a tripod by taking $\{\textit{second column}\} \rightarrow \cdots \rightarrow v_2' \rightarrow v_3 \rightarrow \{\textit{path}\}$ and adding $v_1$. With similar arguments we can conclude there are no bricks in the second column above $v_2''$ and $v_1$ is the only brick of the first column. Finally, we now see that there needs to be a brick in the third column above $v_2''$, otherwise the closure is not a knot by the second case of Lemma~\ref{lemma_links}. If there are at least two bricks of the third column above $v_2'$, we get a tripod by taking $\{\textit{third column}\} \rightarrow \cdots \rightarrow v_3' \rightarrow v_4 \rightarrow \{\textit{path}\}$ and adding $v_2'$ and $v_2$. Otherwise, we can consider $v_1 \rightarrow v_2'' \rightarrow \cdots \rightarrow v_2' \rightarrow v_3 \rightarrow \{\textit{path}\}$ and add the other brick in the third column linked to $v_2'$, which now we know will not be linked to any other brick of the second column (it is also useful to remember that, as stated at the beginning of Case~\ref{case_above_C_a}, $v_3'$, and hence all the bricks of the third column above it, is not connected to the path in the fourth column above $v_4$).
				
					\begin{figure}[h]
						\centering
						\def\svgwidth{.8\columnwidth}			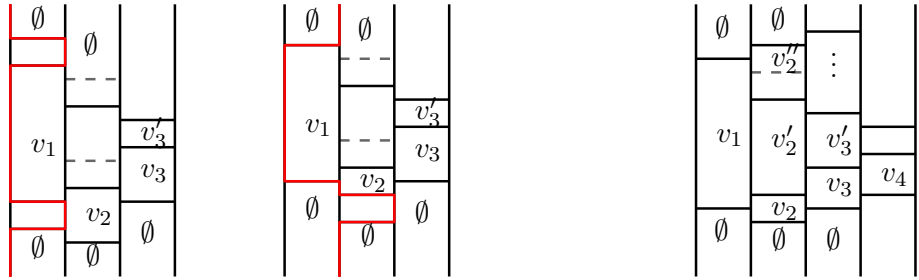
						\caption{Additional diagrams for Case~\ref{case_above_C_a1}.}
						\label{Ab_Ca1_2}
					\end{figure}
				
					\item\label{case_above_C_a2} If there are no bricks in the first column above $v_2$, but $v_2$ is linked to a brick $v_1$ below it, as in the left of Figure \ref{Ab_Ca2}, we can directly conclude that, depending on the number of bricks in the first column, either the closure is not a knot by Lemma~\ref{lemma_links} or we find an appropriate tripod.

					\begin{figure}[h]
						\centering
						\def\svgwidth{\columnwidth}			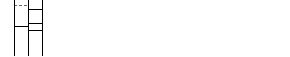
						\caption{Diagrams for Case~\ref{case_above_C_a2}.}
						\label{Ab_Ca2}
					\end{figure}
					
				\end{enumerate}

				\item\label{case_above_C_b} Suppose now that there are no bricks in the second column below $v_3$, which is therefore only linked to a brick $v_2$ above it.
				
				\begin{enumerate}
					\item\label{case_above_C_b1}If in the second column there are bricks both above and below $v_2$, noticing that if there are at least four bricks in the second column we are done, we are only left with the cases of Figure~\ref{Ab_Cb1}. For the leftmost diagram, if there is only one brick in the first column the result is not a knot by Lemma~\ref{lemma_links}, otherwise up to elementary conjugation we get a tripod. In the two central diagrams we directly find a tripod. In the rightmost diagram, if $v_2'$ is not linked to the third column the closure is not a knot by the first case of Lemma~\ref{lemma_links}, otherwise in the third column there is in particular a brick $v_3''$ linked to $v_2$ from above, and we get a tripod by taking $v_1\rightarrow v_2'\rightarrow v_2 \rightarrow v_3 \rightarrow \{\textit{path}\}$ and adding $v_3''$ (again, we use that $v_3'$ is not linked to the path above $v_4$, hence $v_3''$ is also not).
				
					\begin{figure}[h]
						\centering
						\def\svgwidth{.9\columnwidth}			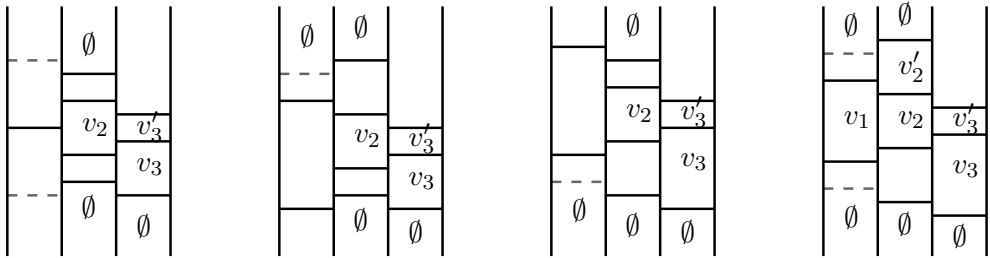
						\caption{Diagrams for Case~\ref{case_above_C_b1}, when $v_3$ and $w_3$ coincide and there is no brick in the second column below $v_3$.}
						\label{Ab_Cb1}
					\end{figure}
				
					\item\label{case_above_C_b2} If in the second column there are only bricks below $v_2$, by minimality of the number of strands there must be a brick $v_3''$ in the third column above $v_2$. If $v_2$ is not linked to the first column, as in the left of Figure~\ref{Ab_Cb2}, we can simply take our original path starting from the first column and add to it $v_3''$. If $v_2$ is linked to the left, notice that by the condition on braid relations it can only be linked to a brick $v_1$ from below; in Figure~\ref{Ab_Cb2} we show that we always get a tripod or a link with at least two components. 
					
					\begin{figure}[h]
						\centering
						\def\svgwidth{\columnwidth}			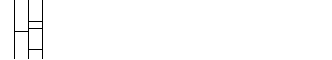
						\caption{Diagrams for Case~\ref{case_above_C_b2}}
						\label{Ab_Cb2}
					\end{figure}
					
					\item\label{case_above_C_b3} Finally, if in the second column there are only bricks above $v_2$, let us consider $v_2'$ the first brick of the second column linked to a brick $v_1$ of the first column (starting from $v_2$ upwards, potentially $v_2' = v_2$). If there is still a brick $v_2''$ above it, up to elementary conjugation on the first column we can assume that $v_1$ is above $v_2'$, as in the left of Figure~\ref{Ab_Cb3}. We now directly see that we can suppose there is only one brick in the first column and that according to whether $v_2''$ is linked to its right or not, we either get a tripod or a link with more than one component by Lemma~\ref{lemma_links}. If there are no more bricks above $v_2'$, we are left with the diagram at the right-hand side of Figure~\ref{Ab_Cb3}. By Lemma~\ref{lemma_braid_index}, we know that there must be bricks above and below $v_1$ and we conclude with an usual argument, according to the number of those bricks.			
				
					\begin{figure}[h]
						\centering
						\def\svgwidth{.5\columnwidth}			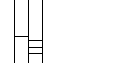
						\caption{Diagrams for Case~\ref{case_above_C_b3}}
						\label{Ab_Cb3}
					\end{figure}
				
			\end{enumerate}
		
			\end{enumerate}
			
		\end{enumerate}
	
	\item\label{case_below} We finally have to treat the case where there is a free brick $w_3'$ below $w_3$, but no brick above $v_3$. Once more, we distinguish according to how $w_3'$ is connected to the path.
	
		\begin{enumerate}
			\item\label{case_below_A} Let us first suppose that $w_3'$ is linked to a brick $w_2$ of the original path in the second column (which, by construction, will also be linked to $v_3$). Then either $v_3$ and $w_3$ are adjacent or they coincide, as in Figure~\ref{Bel_A}.
			
			\begin{figure}[h]
				\centering
				\def\svgwidth{.4\columnwidth}			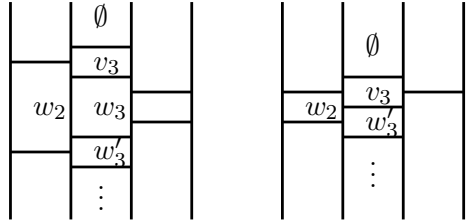
				\caption{The two main possibilities for Case~\ref{case_below_A}}
				\label{Bel_A}
			\end{figure}
			
			\begin{enumerate}
				\item\label{case_below_A_a}	If $v_3$ and $w_3$ are distinct, by construction we furthermore know that $v_3$ and $w_3'$ are not linked to any brick of the fourth column. If $v_3$ is linked to a brick $v_2$ of the second column above $w_2$, we know that in turns $v_2$ needs to be linked to the first column. In this case, we could simply connect the first column to $v_3$ via $v_2$ (thus skipping $w_2$), continue with our original path and add to it $w_3'$ to get a tripod. Similarly, suppose that $w_3'$ is linked to some brick $w_2'$ in the second column below $w_2$. If there is a connection between the first and second columns below $w_2$, the previous argument still applies: we can connect $w_3'$ to the first column bypassing $w_2$, continue with our original path from $w_3$ and connect $v_3$ as isolated leaf of the tripod. Otherwise, all the bricks in the second column below $w_2$ are "free" and can be added to our path. In particular, if there are at least two of them we are done. Moreover, as shown in the left of Figure~\ref{Bel_Aa}, we also directly find a tripod if there are at least two bricks above $w_2$ or if $w_2$ is not connected to the first column. We are then now left with the rightmost diagram of Figure~\ref{Bel_Aa}. Here it is clear that if there are at least two bricks in the first column we find a tripod (potentially after one elementary conjugation), otherwise Lemma~\ref{lemma_braid_index} forces the existence of a brick above $w_2$, in what case the closure is not a knot by Lemma~\ref{lemma_links}.

				\begin{figure}[h]
					\centering
					\def\svgwidth{.75\columnwidth}			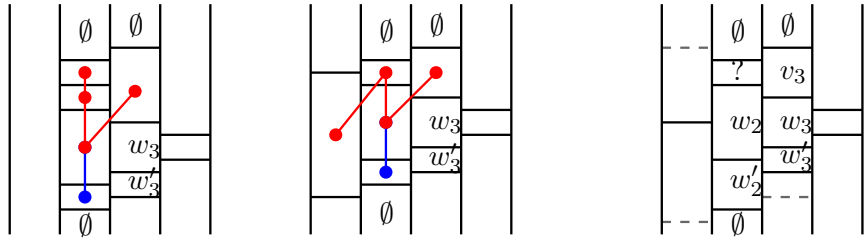
					\caption{First diagrams for Case~\ref{case_below_A_a}, when $w_3'$ is linked to the second column from below.}
					\label{Bel_Aa}
				\end{figure}

				 We can therefore now suppose that there is also no brick of the second column below $w_3'$, as depicted in the left of Figure~\ref{Bel_Aa_2}.
				 If in the second column there are bricks both above and below $w_2$, we are basically in the situation of Case~\ref{case_above_C_b1} (with the appropriate changes in the third column) and the same arguments apply. If there are only bricks above $w_2$, considering $v_2'$ the first brick of the second column linked to a brick $v_1$ of the first column, we get a diagram as in the right of Figure~\ref{Bel_Aa_2}. Note that this is analogous to Case~\ref{case_above_C_b3} and we conclude similarly. The case where there are only bricks below $w_2$ is symmetric.
				 
				 \begin{figure}[h]
				 	\centering
				 	\def\svgwidth{.55\columnwidth}			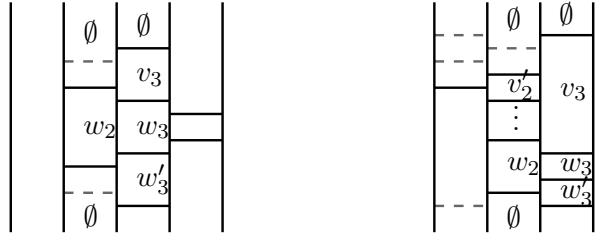
				 	\caption{Final diagrams for Case~\ref{case_below_A_a}.}
				 	\label{Bel_Aa_2}
				 \end{figure}

				\item If $v_3$ and $w_3$ coincide, as in the right-hand side of Figure~\ref{Bel_A}, we know that $w_2$ needs to be linked to a brick $v_1$ in the first column.
				
				\begin{enumerate}
					\item\label{case_below_A_b1} If $v_1$ is above $w_2$, as in the left of Figure~\ref{Bel_Ab1}, we are in a situation very similar to Case~\ref{case_ab-bel_B_a}. First, after removing all the cases where one can directly find a tripod, we can suppose that there are at most two bricks in the first column, both linked to $w_2$, and we know that the brick immediately above $v_2$ (if any) is linked to $v_1$. We are then left with diagrams as in Figure~\ref{Bel_Ab1_2}. In the right-hand side we directly see that the closure is not a knot, while the left-hand side can be solved as in Case~\ref{case_ab-bel_B_a} (compare with Figure~\ref{Ab-Bel_Ba}).
					
					\begin{figure}[h]
						\centering
						\def\svgwidth{\columnwidth}			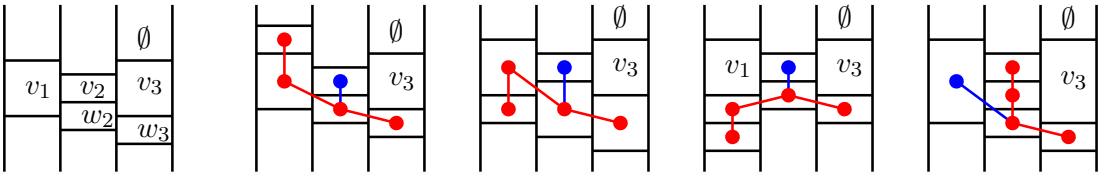
						\caption{First diagrams for Case~\ref{case_below_A_b1}.}
						\label{Bel_Ab1}
					\end{figure}
				
					\begin{figure}[h]
						\centering
						\def\svgwidth{.45\columnwidth}			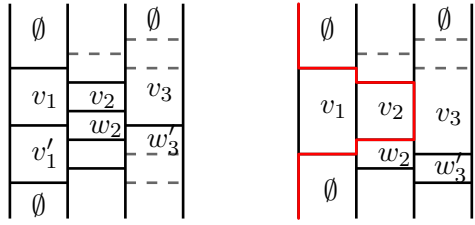
						\caption{Final diagrams for Case~\ref{case_below_A_b1}. Recall that there is no brick in the third column above $v_3$.}
						\label{Bel_Ab1_2}
					\end{figure}
				
					\item\label{case_below_A_b2} If $w_2$ is not linked to any brick in the first column from above, as in the left of Figure~\ref{Bel_Ab2}, after removing some easy cases shown in Figure~\ref{Bel_Ab2}, we can suppose that there is at most one brick in the first column, and we are left with a diagram as in the left-hand side of Figure~\ref{Bel_Ab2_2}. Note that this is similar to Case~\ref{case_ab-bel_A_b1}, compare with the rightmost diagram of Figure~\ref{Ab-Bel_Ab1_2}.
					
					\begin{figure}[h]
						\centering
						\def\svgwidth{\columnwidth}			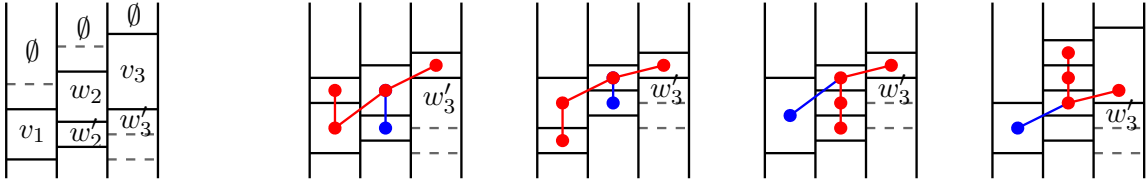
						\caption{First diagrams for Case~\ref{case_below_A_b2}.}
						\label{Bel_Ab2}
					\end{figure}
				
					Now, if $b$ is not linked to the path in the fourth column, as was the case in Case~\ref{case_ab-bel_A_b1}, or the brick denoted by $a$ does not exist, the same argument discussed therein still works. If $b$ is linked to the path in the fourth column from below, one can consider $v_1 \rightarrow a \rightarrow w_2' \rightarrow w_3' \rightarrow b \rightarrow \{\textit{path}\}$ and add $v_3$ to get a tripod. Similarly if $w_3'$ is linked to the path in the fourth column from below. Finally, if $b$ is linked to the path in the fourth column from above but $w_3'$ is not, we are in the case drawn in the right-hand side of Figure~\ref{Bel_Ab2_2}. If in the third column there are no bricks below $a$, one can simply perform an elementary conjugation on the second column to get a brick $v_2$ above $w_2$, take the original path starting from $v_3$ and add to it $w_3' \rightarrow w_2'$ and $v_2$ to obtain a tripod. Finally, if in the third column there is a brick below $a$, in particular $w_2'$ is linked to a brick $b'$ of the third column below $b$. One can hence take $v_1 \rightarrow a \rightarrow w_2' \rightarrow b' \rightarrow \cdots \rightarrow b \rightarrow v_4 \rightarrow \{\textit{path}\}$ (or potentially skipping $w_2'$ if $b'$ is also linked to $a$ ) and connect $v_3$ to $v_4$.
					
					\begin{figure}[h]
						\centering
						\def\svgwidth{.55\columnwidth}			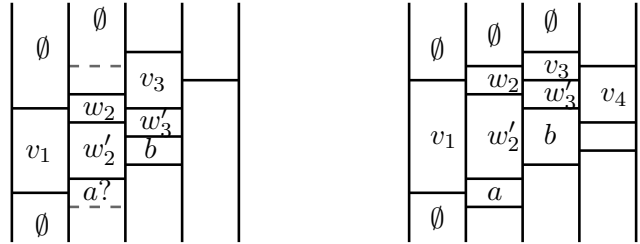
						\caption{Additional diagrams for Case~\ref{case_below_A_b2}.}
						\label{Bel_Ab2_2}
					\end{figure}
				
				\end{enumerate}
		
			\end{enumerate}
		
		\item\label{case_below_B} We now suppose that $w_3'$ is not linked to the original path in the second column (and therefore has to be linked to the path in the fourth column). By construction, we know that $v_3$ is linked to some brick in the second column.

			\begin{enumerate}
				\item\label{case_below_B_a} Assume first that $v_3$ is linked to a brick $w_2$ above it, as in the left of Figure~\ref{Bel_Ba}. Note that the situation is similar to the one analyzed in Case~\ref{case_ab-bel_B}, and many of the arguments discussed therein will apply to the current case. First of all, we know that $w_2$ will be linked to a brick of the first column. If it is linked to a brick $v_1$ above it, as in the right of Figure~\ref{Bel_Ba}, we conclude directly as in Case~\ref{case_ab-bel_B_a} (compare also with the left diagram of Figure~\ref{Bel_Ab1} and the discussion of Case~\ref{case_below_A_b1}).
				
				\begin{figure}[h]
					\centering
					\def\svgwidth{.5\columnwidth}			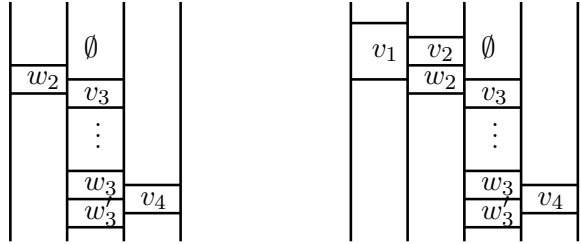
					\caption{First diagrams for Case~\ref{case_below_B_a}}
					\label{Bel_Ba}
				\end{figure}
			
				We can now assume that $w_2$ is only linked to a brick $v_1$ below it, as in the left of Figure~\ref{Bel_Ba_2}.
				By Case~\ref{case_ab-bel_B_b}, we are only left with the two diagrams in the center of Figure~\ref{Bel_Ba_2}, and we furthermore can assume that the brick $w_2'$ is not linked to the original path in the third column below $v_3$ (so no other brick of the second column below $w_2'$ is) and that, as drawn, there is at least one brick $w_2''$ in the second column below $v_1$ (compare with Figure~\ref{Ab-Bel_Bb} and the discussion preceding it).
				
				\begin{figure}[h]
					\centering
					\def\svgwidth{\columnwidth}			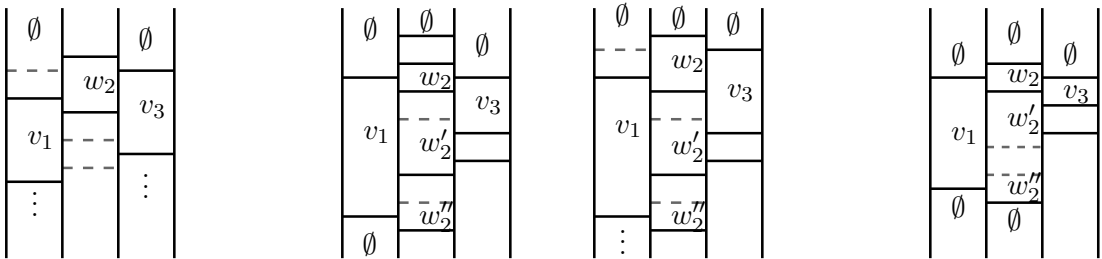
					\caption{Diagrams for Case~\ref{case_below_B_a}}
					\label{Bel_Ba_2}
				\end{figure}
			
				After removing all the cases where one can directly find a tripod, as shown in Figure~\ref{Bel_Ba_3}, we are left with the rightmost diagram of Figure~\ref{Bel_Ba_2}.
				
				\begin{figure}[h]
					\centering
					\def\svgwidth{1.1\columnwidth}			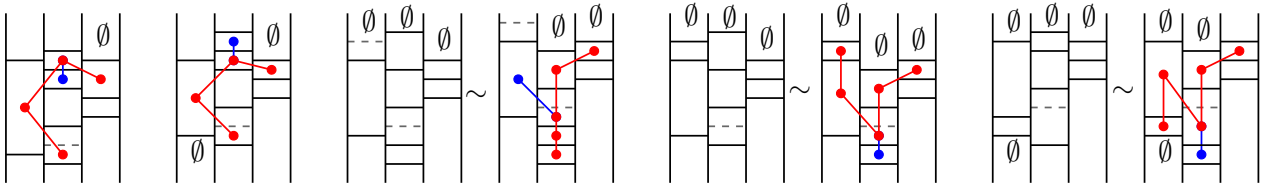
					\caption{Additional diagrams for Case~\ref{case_below_B_a}}
					\label{Bel_Ba_3}
				\end{figure}
				
				But now we observe that there needs to be a brick in the third column below $w_2''$, otherwise the closure is not a knot by Lemma ~\ref{lemma_links}. In particular, $w_2'$ is linked to a brick of the third column below $v_3$. Recalling that $w_2'$ is not linked to the original path in the third column below $v_3$, it follows that either $w_2'$ is linked to $w_3'$ or to some brick below it (in the notation of Figure~\ref{Bel_Ba}).
				
				\begin{enumerate}
					\item \label{case_below_B_a1} Let us first assume that $w_2'$ is linked to $w_3'$, as in Figure~\ref{Bel_Ba1}. Notice that in that case by construction $v_3$ is not linked to the path in the fourth column. If $w_3'$ is only linked to the second column in $w_2'$, as in the left of Figure~\ref{Bel_Ba1}, we can find can simply take $v_1 \rightarrow w_2'' \rightarrow \cdots \rightarrow w_2'\rightarrow w_3''\rightarrow \{\textit{path in the fourth column}\}$ and connect $v_3$ to $w_2'$. Otherwise, we know that there exists at least one brick in the third column below $w_3'$, as in the right of Figure~\ref{Bel_Ba1}. Consider now how $w_3'$ is connected to the path in the fourth column: if it is only connected to $v_4$, all the bricks of the third column below $w_3'$ are free to use and we can take $w_2 \rightarrow w_2' \rightarrow w_3'\rightarrow v_4\rightarrow \{\textit{path}\}$ and connect the brick below $w_3'$ as leaf of the tripod; if $w_3'$ is connected to the path in the fourth column from below, via a brick $w_4$, take instead $v_1 \rightarrow w_2'' \rightarrow \cdots \rightarrow w_3'\rightarrow w_4''\rightarrow \{\textit{path}\}$ and connect $w_3$ as a leaf.
					
					\begin{figure}[h]
						\centering
						\def\svgwidth{.5\columnwidth}			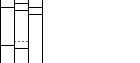
						\caption{Diagrams for case~\ref{case_below_B_a1}.}
						\label{Bel_Ba1}
					\end{figure}
					
					\item\label{case_below_B_a2} If $w_2'$ is linked to a brick $w_3''$ below $w_3'$, we have one of the diagrams of Figure~\ref{Bel_Ba2}. If $v_3$ and $w_3$ are distinct, as in the left of Figure~\ref{Bel_Ba2}, we recognize the diagram of Figure~\ref{Nobricks}, and the argument discussed there applies to the current setting. If $v_3$ and $w_3$ coincide, we have the diagram on the right of Figure~\ref{Bel_Ba2}. Once again, we consider how $w_3'$ is connected to the path in the fourth column. If $w_3'$ is linked to the path in the fourth column under $v_4$, we can simply take $ v_1 \rightarrow w_2 \rightarrow v_3 \rightarrow w_3' \rightarrow \{\textit{path}\}$ and add a brick of the fourth column (which will be at most $v_4'$). Finally, if $w_3'$ is not linked to the path in the fourth column below $v_4$, then all the bricks in the third column under $w_3'$ also are not, and can hence be freely used. If there is still at least one brick in the third column under $w_3''$, we can take $w_2 \rightarrow w_2' \rightarrow w_3'' \rightarrow \cdots \rightarrow w_3' \rightarrow v_4 \rightarrow \{\textit{path}\}$ and add a brick below $w_3''$ to get a tripod. If $w_3''$ is the last brick of the third column, in particular it is not linked to any of the bricks below $w_2'$, so we can take $v_1 \rightarrow w_2'' \rightarrow \cdots \rightarrow w_2' \rightarrow v_3 \rightarrow \{\textit{path}\}$ and connect $w_3''$ to $w_2'$.
					
					\begin{figure}[h]
						\centering
						\def\svgwidth{.5\columnwidth}			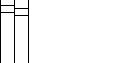
						\caption{Diagrams for case~\ref{case_below_B_a2}.}
						\label{Bel_Ba2}
					\end{figure}
				
				\end{enumerate}

				\item\label{case_below_B_b} We can now suppose that $v_3$ is only linked to a brick $w_2$ of the second column below it. In particular, our original path was passing by $w_2$, which is therefore not linked to $w_3'$. We now get the diagrams of Figure~\ref{Bel_Bb}. In the left-hand side, where $v_3$ and $w_3$ are distinct, we end up with a diagram similar to Figure~\ref{Nobricks} and the exact same arguments apply. Suppose now that $v_3$ and $w_3$ coincide, as in the right-hand side of Figure~\ref{Bel_Bb}. If $w_3'$ is linked to the path in the fourth column under $v_4$, we can simply take $ w_2 \rightarrow v_3 \rightarrow w_3' \rightarrow \{\textit{path}\}$ and add a brick of the fourth column (which will be at most $v_4'$). Otherwise, we are in a situation perfectly symmetric to Case~\ref{case_above_C_b}, in particular as in Figures~\ref{Ab_Cb1}, \ref{Ab_Cb2} and~\ref{Ab_Cb3}, and again the same arguments apply.
		
				\begin{figure}[h]
					\centering
					\def\svgwidth{.45\columnwidth}			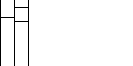
					\caption{Diagrams for case~\ref{case_below_B_a2}.}
					\label{Bel_Bb}
				\end{figure}
			
			\end{enumerate}
		
		\end{enumerate}
		
	\end{enumerate}
	
\end{proof}

\bigskip

We still have to consider the braids of intermediate positive braid index. One could probably study those by hands, in a similar way to Prop.~\ref{prop 3-braids} and Prop.~\ref{prop big braid index}, but the computations would quickly get too complicated. Instead, we will treat them by directly applying Proposition~\ref{prop 3-braids}, at the cost of loosing some low genus cases.

\begin{prop}\label{prop_inter_braid_index}
	Let $\beta$ be a prime positive braid on $4\leq N \leq 10$ strands whose closure is a knot not of type $A_n$. Suppose that $\beta$ has genus $g(\beta) > 4(N-1)$. Then there exists a family of curves on $\Sigma_{\beta}$ that is an $E$-arboreal spanning configuration on a subsurface of genus at least $5$.
\end{prop}

The curves appearing in Proposition~\ref{prop_inter_braid_index} will not necessarily be vertices of the intersection graph, but we might need to do some "change of basis", i.e. modify some of the curves by applying appropriate Dehn twists. This will change the intersection pattern of the curves in question, but not the subsurface they span nor the subgroup that the corresponding Dehn twists generate in $\mathrm{Mod}(\Sigma_{\beta})$. 

\begin{proof}		
	Let $\beta$ be such a positive braid. Since $g(\beta) > 4(N-1)$, there exists $1 \leq i \leq N-2$ such that the subword induced by all the generators $\sigma_i$ and $\sigma_{i+1}$ has first Betti number at least $12$, when seen as a $3$-braid. Let us denote this subword by $\beta_{i,i+1}$. By Proposition~\ref{prop 3-braids}, either $\beta_{i,i+1}$ is positively isotopic to a $3$-braid $\beta_{i,i+1}'$ containing the required spanning configuration, or it is of type $A_n$ or $D_n$ (the other finitely many exceptions have first Betti number $11$).
	
	In the first case, the required positive braid isotopy might not be realizable when $\beta_{i,i+1}$ is seen as a subword of $\beta$. However, since at the level of curves the effect of braid relations and elementary conjugations is obtained by Dehn twists, we can still find a family of curves in $\Sigma_{\beta_{i,i+1}} \subset \Sigma_{\beta}$ whose intersection pattern is equal to the linking graph of $\beta_{i,i+1}'$, and the result follows.
	
	If $\beta_{i,i+1}$ is of type $A_n$, since there are only three strands one can directly verify that up to elementary conjugation its linking graph is a path. We can therefore apply Lemma~\ref{lemma_braid_index} to $\beta$ and reduce it to a braid with less strands.
	
	If $\beta_{i,i+1}$ is of type $D_n$, up to elementary conjugation and symmetry it is of one of three forms: $\sigma_i^{n-3}\sigma_{i+1}^2\sigma_i\sigma_{i+1}^2$, $\sigma_i^{n-2}\sigma_{i+1}\sigma_i^2\sigma_{i+1}$ or $\sigma_i^a\sigma_{i+1}\sigma_i\sigma_{i+1}^b\sigma_i\sigma_{i+1}$ with, $a+b = n-2$. This follows from a direct computation, or can be seen by applying the classification of checkerboard graphs of type $D_n$ contained in Lucas Vilanova's PhD thesis \cite{vilanova_positive_2020}. In all the cases one can see that, if the closure is connected, we can always add a brick in a neighbouring column and find the required subtree. We will do it for $\beta_{i,i+1} = \sigma_i^{n-3}\sigma_{i+1}^2\sigma_i\sigma_{i+1}^2$, the others are analogous. In this case, we know that $i < N-2$, otherwise the closure is not a knot by Lemma~\ref{lemma_links}. Since $\beta$ is prime, its intersection graph is connected, so at least one of the three bricks in the $i+1$-th column needs to be linked to its right. After removing the cases where one directly finds an appropriate subtree, we are left with one of the three cases of Figure~\ref{figure_intermediate_Dn}. The first one is excluded since the closure is not a knot; in the second one we can find a subtree after braid relation, as shown in the Figure; for the third one, up to elementary conjugation we can suppose that there are no generators $\sigma_{i+2}$ above the last occurrence of $\sigma_{i+1}$. Now we see that if there are at least two bricks in the $i+2$-th column we are done, otherwise either the closure is not a knot (if $i+2 = N-1$) or we can still add one brick further to the right and again find the required subtree.
	
	\begin{figure}
		\centering
		\def\svgwidth{.8\columnwidth}			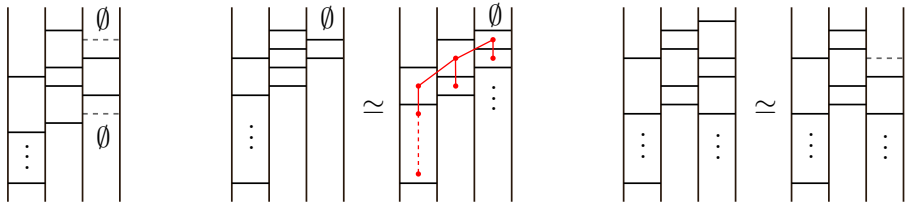
		\caption{The columns $i,i+1$ and $i+2$ of a braid $\beta$ such that $\beta_{i,i+1} = \sigma_i^{n-3}\sigma_{i+1}^2\sigma_i\sigma_{i+1}^2$.}
		\label{figure_intermediate_Dn}
	\end{figure}
\end{proof}

Everything is now ready to prove our main Theorem.

\begin{proof}[Proof of Theorem \ref{theorem_framed}]
	Let $\beta$ be a prime positive braid not of type $A_n$ and whose closure is a knot. We want to prove that $\mathit{MG}(\beta) = \mathrm{Mod}(\Sigma_{\beta},\phi_{\beta})$ by using Proposition~\ref{prop_finite_generation}. Let $V = \{\gamma_1,\cdots,\gamma_{2g}\}$ be the family of standard curves on $\Sigma_{\beta}$ corresponding to the vertices of the linking graph of $\beta$. In Propositions~\ref{prop 3-braids},~\ref{prop big braid index} and~\ref{prop_inter_braid_index} we have constructed the starting $E$-arboreal spanning configuration of genus $h \geq 5$ for all but finitely many such prime positive braids. In general, this is obtained by taking a subfamily of curves $V_0' \subset V$ and potentially modifying some of them by applying Dehn twists around other curves of $V_0'$, obtaining a family $V_0$ of curves in $\Sigma_{\beta}$. In particular, the subsurface spanned by $V_0$ is the same as the subsurface spanned by $V_0'$. It is now clear that the remaining curves of $V \setminus V_0'$ can be attached in an order that respects the definition of $h$-assemblage, so that
	$$ \mathrm{Mod}(\Sigma_{\beta},\phi_{\beta}) = \langle T_c \mid c\in V_0 \cup (V\setminus V_0') \rangle = \langle T_c \mid c \in V \rangle = \mathit{MG}(\beta).$$
\end{proof}

\begin{rem}\label{remark_links_sing}
	In fact, our proof of Theorem~\ref{theorem_framed} also applies to many links. Indeed, the requirement of the closure of $\beta$ being a knot was uniquely used to exclude links as in Lemma~\ref{lemma_links}: all these have one unknotted component whose total linking number with the other components is precisely $2$. In particular, the proof works without problems for links whose components are all knotted or whose pairwise linking numbers are all big enough.
	
	Interestingly, this is essentially always the case in the special class of links of singularities, if we exclude the special families $A_n$ and $D_n$. In what follows, the reader can refer to \cite{brieskorn_plane_1986} for the background material on plane curve singularities. If $f_1$ and $f_2$ are irreducible singularities with associated knots $K_1$ and $K_2$, then the link of $f=f_1f_2$ is $L(f)= K_1 \cup K_2$, and the linking number $lk(K_1,K_2)$ equals the intersection multiplicity of the two branches. It follows that in the link of a singularity all linking numbers are strictly positive. Now, let $f$ be a singularity whose link has a component which is unknotted and has total linking number with the other components equal to $2$, as in Lemma~\ref{lemma_links}. By the previous discussion, $f$ has at most three branches. Suppose first that $f=f_1f_2$ has only two branches, and $L(f) = K_1 \cup K_2$. Since one component is the unknot and the multiplicity of a singularity equals the braid index of the associated link by \cite{williams_braid_1988}, we can assume that $f_2 = y+x\tilde{f}(x,y)$. Let now $m$ be the multiplicity and $y = g(x^{\frac{1}{m}})$ the Puiseux series of $f_1$, we obtain $2 = lk(K_1,K_2) = \mathrm{ord} (g(t)+t^m\tilde{f}(t^m,g(t))) \geq m$, from which we conclude that $K_1$ has braid index at most $2$. Finally, since the link of a reducible singularity is determined by the components and the pairwise linking numbers, and all the possible pairs of a positive $2$-braid and an unknot with linking number $2$ are realized by singularities of type $A_n$ or $D_n$, it follows that $f$ belongs to one of those two families. Similarly, if $f$ has three branches one can conclude that all the components of $L(f)$ are unknotted, so that the link is determined by the triple of linking numbers (where two of the linking numbers are now equal to $1$). Since all such triples are realized by singularities of type $D_n$, $f$ must belong to this family.
	Therefore, up to finitely many low genus exceptions, we completely recover the main result of \cite{portilla_cuadrado_vanishing_2021}, saying that the geometric monodromy group of a singularity not of type $A_n$ and $D_n$ is a framed mapping class group.
	
\end{rem}

\begin{rem}\label{remark_links}
	
	In contrast to the case of singularities, it does not seem possible to extend the proof to all positive braid links. Even excluding the two exceptional families $A_n$ and $D_n$, there are other infinite families, both with bounded and unbounded braid index, that most likely do not contain an $E_6$. For example, we could not find such subtrees for the braids $\beta_n = \sigma_1\sigma_2^2\sigma_1\sigma_2^{n-4}\sigma_3\sigma_2^2\sigma_3 \in B_4^+$, whose linking graph is the extended Dynkin diagram $\tilde{D}_n$, nor for $\beta_N =( \sigma_1\cdots\sigma_N\sigma_N\cdots\sigma_1)^2 \in B_{N+1}^+$. We do not know whether the corresponding monodromy groups are equal to the whole framed mapping class group.
\end{rem}

\printbibliography

\bigskip
\bigskip
\bigskip
\noindent
\texttt{livio.ferretti@unige.ch}

\end{document}

%% file: Figures/first_example.pdf_tex
\begingroup%
  \makeatletter%
  \providecommand\color[2][]{%
    \errmessage{(Inkscape) Color is used for the text in Inkscape, but the package 'color.sty' is not loaded}%
    \renewcommand\color[2][]{}%
  }%
  \providecommand\transparent[1]{%
    \errmessage{(Inkscape) Transparency is used (non-zero) for the text in Inkscape, but the package 'transparent.sty' is not loaded}%
    \renewcommand\transparent[1]{}%
  }%
  \providecommand\rotatebox[2]{#2}%
  \newcommand*\fsize{\dimexpr\f@size pt\relax}%
  \newcommand*\lineheight[1]{\fontsize{\fsize}{#1\fsize}\selectfont}%
  \ifx\svgwidth\undefined%
    \setlength{\unitlength}{568.14783772bp}%
    \ifx\svgscale\undefined%
      \relax%
    \else%
      \setlength{\unitlength}{\unitlength * \real{\svgscale}}%
    \fi%
  \else%
    \setlength{\unitlength}{\svgwidth}%
  \fi%
  \global\let\svgwidth\undefined%
  \global\let\svgscale\undefined%
  \makeatother%
  \begin{picture}(1,0.89066)%
    \lineheight{1}%
    \setlength\tabcolsep{0pt}%
    \put(0,0){\includegraphics[width=\unitlength,page=1]{first_example.pdf}}%
    \put(0.19379029,0.33269505){\color[rgb]{0,0,1}\makebox(0,0)[lt]{\lineheight{1.25}\smash{\begin{tabular}[t]{l}$\gamma_3$\\\end{tabular}}}}%
    \put(0.32515017,0.0815163){\color[rgb]{1,0,1}\makebox(0,0)[lt]{\lineheight{1.25}\smash{\begin{tabular}[t]{l}$\gamma_4$\end{tabular}}}}%
    \put(0,0){\includegraphics[width=\unitlength,page=2]{first_example.pdf}}%
    \put(0.05899714,0.45313329){\color[rgb]{1,0,0}\makebox(0,0)[lt]{\lineheight{1.25}\smash{\begin{tabular}[t]{l}$\gamma_2$\\\end{tabular}}}}%
    \put(0.05911343,0.2008258){\color[rgb]{0,0.50196078,0}\makebox(0,0)[lt]{\lineheight{1.25}\smash{\begin{tabular}[t]{l}$\gamma_1$\\\end{tabular}}}}%
    \put(0,0){\includegraphics[width=\unitlength,page=3]{first_example.pdf}}%
  \end{picture}%
\endgroup%

%% file: Figures/braid_rel.pdf_tex
\begingroup%
  \makeatletter%
  \providecommand\color[2][]{%
    \errmessage{(Inkscape) Color is used for the text in Inkscape, but the package 'color.sty' is not loaded}%
    \renewcommand\color[2][]{}%
  }%
  \providecommand\transparent[1]{%
    \errmessage{(Inkscape) Transparency is used (non-zero) for the text in Inkscape, but the package 'transparent.sty' is not loaded}%
    \renewcommand\transparent[1]{}%
  }%
  \providecommand\rotatebox[2]{#2}%
  \newcommand*\fsize{\dimexpr\f@size pt\relax}%
  \newcommand*\lineheight[1]{\fontsize{\fsize}{#1\fsize}\selectfont}%
  \ifx\svgwidth\undefined%
    \setlength{\unitlength}{795.87858906bp}%
    \ifx\svgscale\undefined%
      \relax%
    \else%
      \setlength{\unitlength}{\unitlength * \real{\svgscale}}%
    \fi%
  \else%
    \setlength{\unitlength}{\svgwidth}%
  \fi%
  \global\let\svgwidth\undefined%
  \global\let\svgscale\undefined%
  \makeatother%
  \begin{picture}(1,0.2981536)%
    \lineheight{1}%
    \setlength\tabcolsep{0pt}%
    \put(0,0){\includegraphics[width=\unitlength,page=1]{braid_rel.pdf}}%
    \put(0.12149063,0.00177794){\makebox(0,0)[lt]{\lineheight{1.25}\smash{\begin{tabular}[t]{l}$\Sigma_\alpha$\end{tabular}}}}%
    \put(0.85237517,0.00177794){\makebox(0,0)[lt]{\lineheight{1.25}\smash{\begin{tabular}[t]{l}$\Sigma_\beta$\end{tabular}}}}%
    \put(0.12463067,0.08627654){\makebox(0,0)[lt]{\lineheight{1.25}\smash{\begin{tabular}[t]{l}$\omega$\end{tabular}}}}%
    \put(0.4881193,0.08469535){\makebox(0,0)[lt]{\lineheight{1.25}\smash{\begin{tabular}[t]{l}$\omega$\end{tabular}}}}%
    \put(0.85470385,0.08619701){\makebox(0,0)[lt]{\lineheight{1.25}\smash{\begin{tabular}[t]{l}$\omega$\end{tabular}}}}%
  \end{picture}%
\endgroup%

%% file: Figures/isotopy_divide_braid.pdf_tex
\begingroup%
  \makeatletter%
  \providecommand\color[2][]{%
    \errmessage{(Inkscape) Color is used for the text in Inkscape, but the package 'color.sty' is not loaded}%
    \renewcommand\color[2][]{}%
  }%
  \providecommand\transparent[1]{%
    \errmessage{(Inkscape) Transparency is used (non-zero) for the text in Inkscape, but the package 'transparent.sty' is not loaded}%
    \renewcommand\transparent[1]{}%
  }%
  \providecommand\rotatebox[2]{#2}%
  \newcommand*\fsize{\dimexpr\f@size pt\relax}%
  \newcommand*\lineheight[1]{\fontsize{\fsize}{#1\fsize}\selectfont}%
  \ifx\svgwidth\undefined%
    \setlength{\unitlength}{91.05083796bp}%
    \ifx\svgscale\undefined%
      \relax%
    \else%
      \setlength{\unitlength}{\unitlength * \real{\svgscale}}%
    \fi%
  \else%
    \setlength{\unitlength}{\svgwidth}%
  \fi%
  \global\let\svgwidth\undefined%
  \global\let\svgscale\undefined%
  \makeatother%
  \begin{picture}(1,0.25217375)%
    \lineheight{1}%
    \setlength\tabcolsep{0pt}%
    \put(0,0){\includegraphics[width=\unitlength,page=1]{isotopy_divide_braid.pdf}}%
    \put(0.20547107,0.12608688){\makebox(0,0)[lt]{\lineheight{1.25}\smash{\begin{tabular}[t]{l}$\underset{(1)}{\simeq}$\end{tabular}}}}%
    \put(0.48566414,0.12608688){\makebox(0,0)[lt]{\lineheight{1.25}\smash{\begin{tabular}[t]{l}$\underset{(2a)}{\simeq}$\end{tabular}}}}%
    \put(0.76585718,0.12608688){\makebox(0,0)[lt]{\lineheight{1.25}\smash{\begin{tabular}[t]{l}$\underset{(2b)}{\simeq}$\end{tabular}}}}%
  \end{picture}%
\endgroup%

%% file: Figures/example_surface_divide.pdf_tex
\begingroup%
  \makeatletter%
  \providecommand\color[2][]{%
    \errmessage{(Inkscape) Color is used for the text in Inkscape, but the package 'color.sty' is not loaded}%
    \renewcommand\color[2][]{}%
  }%
  \providecommand\transparent[1]{%
    \errmessage{(Inkscape) Transparency is used (non-zero) for the text in Inkscape, but the package 'transparent.sty' is not loaded}%
    \renewcommand\transparent[1]{}%
  }%
  \providecommand\rotatebox[2]{#2}%
  \newcommand*\fsize{\dimexpr\f@size pt\relax}%
  \newcommand*\lineheight[1]{\fontsize{\fsize}{#1\fsize}\selectfont}%
  \ifx\svgwidth\undefined%
    \setlength{\unitlength}{164.7039819bp}%
    \ifx\svgscale\undefined%
      \relax%
    \else%
      \setlength{\unitlength}{\unitlength * \real{\svgscale}}%
    \fi%
  \else%
    \setlength{\unitlength}{\svgwidth}%
  \fi%
  \global\let\svgwidth\undefined%
  \global\let\svgscale\undefined%
  \makeatother%
  \begin{picture}(1,0.2779894)%
    \lineheight{1}%
    \setlength\tabcolsep{0pt}%
    \put(0,0){\includegraphics[width=\unitlength,page=1]{example_surface_divide.pdf}}%
    \put(3.22745149,0.09698234){\makebox(0,0)[lt]{\lineheight{1.25}\smash{\begin{tabular}[t]{l}$\underset{(1)}{\simeq}$\end{tabular}}}}%
    \put(3.38234642,0.09698234){\makebox(0,0)[lt]{\lineheight{1.25}\smash{\begin{tabular}[t]{l}$\underset{(2a)}{\simeq}$\end{tabular}}}}%
    \put(3.53724108,0.09698234){\makebox(0,0)[lt]{\lineheight{1.25}\smash{\begin{tabular}[t]{l}$\underset{(2b)}{\simeq}$\end{tabular}}}}%
    \put(0,0){\includegraphics[width=\unitlength,page=2]{example_surface_divide.pdf}}%
  \end{picture}%
\endgroup%

%% file: Figures/3braids/m=5.pdf_tex
\begingroup%
  \makeatletter%
  \providecommand\color[2][]{%
    \errmessage{(Inkscape) Color is used for the text in Inkscape, but the package 'color.sty' is not loaded}%
    \renewcommand\color[2][]{}%
  }%
  \providecommand\transparent[1]{%
    \errmessage{(Inkscape) Transparency is used (non-zero) for the text in Inkscape, but the package 'transparent.sty' is not loaded}%
    \renewcommand\transparent[1]{}%
  }%
  \providecommand\rotatebox[2]{#2}%
  \newcommand*\fsize{\dimexpr\f@size pt\relax}%
  \newcommand*\lineheight[1]{\fontsize{\fsize}{#1\fsize}\selectfont}%
  \ifx\svgwidth\undefined%
    \setlength{\unitlength}{56.18425751bp}%
    \ifx\svgscale\undefined%
      \relax%
    \else%
      \setlength{\unitlength}{\unitlength * \real{\svgscale}}%
    \fi%
  \else%
    \setlength{\unitlength}{\svgwidth}%
  \fi%
  \global\let\svgwidth\undefined%
  \global\let\svgscale\undefined%
  \makeatother%
  \begin{picture}(1,0.56955456)%
    \lineheight{1}%
    \setlength\tabcolsep{0pt}%
    \put(0,0){\includegraphics[width=\unitlength,page=1]{m=5.pdf}}%
    \put(0.25081983,0.28477728){\makebox(0,0)[lt]{\lineheight{1.25}\smash{\begin{tabular}[t]{l}$\simeq$\end{tabular}}}}%
    \put(0.67798573,0.28477728){\makebox(0,0)[lt]{\lineheight{1.25}\smash{\begin{tabular}[t]{l}$\simeq$\end{tabular}}}}%
  \end{picture}%
\endgroup%

%% file: Figures/3braids/m=4_bgeq2_second.pdf_tex
\begingroup%
  \makeatletter%
  \providecommand\color[2][]{%
    \errmessage{(Inkscape) Color is used for the text in Inkscape, but the package 'color.sty' is not loaded}%
    \renewcommand\color[2][]{}%
  }%
  \providecommand\transparent[1]{%
    \errmessage{(Inkscape) Transparency is used (non-zero) for the text in Inkscape, but the package 'transparent.sty' is not loaded}%
    \renewcommand\transparent[1]{}%
  }%
  \providecommand\rotatebox[2]{#2}%
  \newcommand*\fsize{\dimexpr\f@size pt\relax}%
  \newcommand*\lineheight[1]{\fontsize{\fsize}{#1\fsize}\selectfont}%
  \ifx\svgwidth\undefined%
    \setlength{\unitlength}{88.18425751bp}%
    \ifx\svgscale\undefined%
      \relax%
    \else%
      \setlength{\unitlength}{\unitlength * \real{\svgscale}}%
    \fi%
  \else%
    \setlength{\unitlength}{\svgwidth}%
  \fi%
  \global\let\svgwidth\undefined%
  \global\let\svgscale\undefined%
  \makeatother%
  \begin{picture}(1,0.31751699)%
    \lineheight{1}%
    \setlength\tabcolsep{0pt}%
    \put(0,0){\includegraphics[width=\unitlength,page=1]{m=4_bgeq2_second.pdf}}%
    \put(0.13712342,0.15875848){\color[rgb]{0,0,0}\makebox(0,0)[lt]{\lineheight{1.25}\smash{\begin{tabular}[t]{l}$\simeq$\end{tabular}}}}%
    \put(0.36392126,0.15875848){\makebox(0,0)[lt]{\lineheight{1.25}\smash{\begin{tabular}[t]{l}$\simeq$\end{tabular}}}}%
    \put(0.5907191,0.15875848){\makebox(0,0)[lt]{\lineheight{1.25}\smash{\begin{tabular}[t]{l}$\simeq$\end{tabular}}}}%
    \put(0.81751695,0.15875848){\makebox(0,0)[lt]{\lineheight{1.25}\smash{\begin{tabular}[t]{l}$\simeq$\end{tabular}}}}%
  \end{picture}%
\endgroup%

%% file: Figures/3braids/m=4_b=1.pdf_tex
\begingroup%
  \makeatletter%
  \providecommand\color[2][]{%
    \errmessage{(Inkscape) Color is used for the text in Inkscape, but the package 'color.sty' is not loaded}%
    \renewcommand\color[2][]{}%
  }%
  \providecommand\transparent[1]{%
    \errmessage{(Inkscape) Transparency is used (non-zero) for the text in Inkscape, but the package 'transparent.sty' is not loaded}%
    \renewcommand\transparent[1]{}%
  }%
  \providecommand\rotatebox[2]{#2}%
  \newcommand*\fsize{\dimexpr\f@size pt\relax}%
  \newcommand*\lineheight[1]{\fontsize{\fsize}{#1\fsize}\selectfont}%
  \ifx\svgwidth\undefined%
    \setlength{\unitlength}{104.18425751bp}%
    \ifx\svgscale\undefined%
      \relax%
    \else%
      \setlength{\unitlength}{\unitlength * \real{\svgscale}}%
    \fi%
  \else%
    \setlength{\unitlength}{\svgwidth}%
  \fi%
  \global\let\svgwidth\undefined%
  \global\let\svgscale\undefined%
  \makeatother%
  \begin{picture}(1,0.28795137)%
    \lineheight{1}%
    \setlength\tabcolsep{0pt}%
    \put(0,0){\includegraphics[width=\unitlength,page=1]{m=4_b=1.pdf}}%
    \put(0.46160648,0.13437731){\makebox(0,0)[lt]{\lineheight{1.25}\smash{\begin{tabular}[t]{l}$\simeq$\end{tabular}}}}%
    \put(0.65357406,0.13437731){\makebox(0,0)[lt]{\lineheight{1.25}\smash{\begin{tabular}[t]{l}$\simeq$\end{tabular}}}}%
    \put(0.84554158,0.13437731){\makebox(0,0)[lt]{\lineheight{1.25}\smash{\begin{tabular}[t]{l}$\simeq$\end{tabular}}}}%
    \put(0,0){\includegraphics[width=\unitlength,page=2]{m=4_b=1.pdf}}%
  \end{picture}%
\endgroup%

%% file: Figures/3braids/m=3_a11.pdf_tex
\begingroup%
  \makeatletter%
  \providecommand\color[2][]{%
    \errmessage{(Inkscape) Color is used for the text in Inkscape, but the package 'color.sty' is not loaded}%
    \renewcommand\color[2][]{}%
  }%
  \providecommand\transparent[1]{%
    \errmessage{(Inkscape) Transparency is used (non-zero) for the text in Inkscape, but the package 'transparent.sty' is not loaded}%
    \renewcommand\transparent[1]{}%
  }%
  \providecommand\rotatebox[2]{#2}%
  \newcommand*\fsize{\dimexpr\f@size pt\relax}%
  \newcommand*\lineheight[1]{\fontsize{\fsize}{#1\fsize}\selectfont}%
  \ifx\svgwidth\undefined%
    \setlength{\unitlength}{86.73186493bp}%
    \ifx\svgscale\undefined%
      \relax%
    \else%
      \setlength{\unitlength}{\unitlength * \real{\svgscale}}%
    \fi%
  \else%
    \setlength{\unitlength}{\svgwidth}%
  \fi%
  \global\let\svgwidth\undefined%
  \global\let\svgscale\undefined%
  \makeatother%
  \begin{picture}(1,0.34589368)%
    \lineheight{1}%
    \setlength\tabcolsep{0pt}%
    \put(0,0){\includegraphics[width=\unitlength,page=1]{m=3_a11.pdf}}%
    \put(-0.00095706,0.17950431){\makebox(0,0)[lt]{\lineheight{1.25}\smash{\begin{tabular}[t]{l}$T(1,k,9-k) = $\end{tabular}}}}%
    \put(0.32943112,0.14090019){\makebox(0,0)[lt]{\lineheight{1.25}\smash{\begin{tabular}[t]{l}$k$\end{tabular}}}}%
    \put(0.44819731,0.13971777){\makebox(0,0)[lt]{\lineheight{1.25}\smash{\begin{tabular}[t]{l}$9-k$\end{tabular}}}}%
  \end{picture}%
\endgroup%

%% file: Figures/3braids/m=3_a8_b5.pdf_tex
\begingroup%
  \makeatletter%
  \providecommand\color[2][]{%
    \errmessage{(Inkscape) Color is used for the text in Inkscape, but the package 'color.sty' is not loaded}%
    \renewcommand\color[2][]{}%
  }%
  \providecommand\transparent[1]{%
    \errmessage{(Inkscape) Transparency is used (non-zero) for the text in Inkscape, but the package 'transparent.sty' is not loaded}%
    \renewcommand\transparent[1]{}%
  }%
  \providecommand\rotatebox[2]{#2}%
  \newcommand*\fsize{\dimexpr\f@size pt\relax}%
  \newcommand*\lineheight[1]{\fontsize{\fsize}{#1\fsize}\selectfont}%
  \ifx\svgwidth\undefined%
    \setlength{\unitlength}{116.18427277bp}%
    \ifx\svgscale\undefined%
      \relax%
    \else%
      \setlength{\unitlength}{\unitlength * \real{\svgscale}}%
    \fi%
  \else%
    \setlength{\unitlength}{\svgwidth}%
  \fi%
  \global\let\svgwidth\undefined%
  \global\let\svgscale\undefined%
  \makeatother%
  \begin{picture}(1,0.55084921)%
    \lineheight{1}%
    \setlength\tabcolsep{0pt}%
    \put(0,0){\includegraphics[width=\unitlength,page=1]{m=3_a8_b5.pdf}}%
    \put(0.1083808,0.43035085){\makebox(0,0)[lt]{\lineheight{1.25}\smash{\begin{tabular}[t]{l}$\simeq$\end{tabular}}}}%
    \put(0.28052109,0.43035085){\makebox(0,0)[lt]{\lineheight{1.25}\smash{\begin{tabular}[t]{l}$\simeq$\end{tabular}}}}%
    \put(0.69365774,0.43035094){\makebox(0,0)[lt]{\lineheight{1.25}\smash{\begin{tabular}[t]{l}$\simeq$\end{tabular}}}}%
    \put(0.86579812,0.43035094){\makebox(0,0)[lt]{\lineheight{1.25}\smash{\begin{tabular}[t]{l}$\simeq$\end{tabular}}}}%
    \put(0,0){\includegraphics[width=\unitlength,page=2]{m=3_a8_b5.pdf}}%
    \put(0.22887903,0.12049827){\makebox(0,0)[lt]{\lineheight{1.25}\smash{\begin{tabular}[t]{l}$\simeq$\end{tabular}}}}%
    \put(0.40101936,0.12049827){\makebox(0,0)[lt]{\lineheight{1.25}\smash{\begin{tabular}[t]{l}$\simeq$\end{tabular}}}}%
    \put(0.5731597,0.12049827){\makebox(0,0)[lt]{\lineheight{1.25}\smash{\begin{tabular}[t]{l}$\simeq$\end{tabular}}}}%
    \put(0.74530008,0.12049827){\makebox(0,0)[lt]{\lineheight{1.25}\smash{\begin{tabular}[t]{l}$\simeq$\end{tabular}}}}%
  \end{picture}%
\endgroup%

%% file: Figures/3braids/m=3_a8_b4.pdf_tex
\begingroup%
  \makeatletter%
  \providecommand\color[2][]{%
    \errmessage{(Inkscape) Color is used for the text in Inkscape, but the package 'color.sty' is not loaded}%
    \renewcommand\color[2][]{}%
  }%
  \providecommand\transparent[1]{%
    \errmessage{(Inkscape) Transparency is used (non-zero) for the text in Inkscape, but the package 'transparent.sty' is not loaded}%
    \renewcommand\transparent[1]{}%
  }%
  \providecommand\rotatebox[2]{#2}%
  \newcommand*\fsize{\dimexpr\f@size pt\relax}%
  \newcommand*\lineheight[1]{\fontsize{\fsize}{#1\fsize}\selectfont}%
  \ifx\svgwidth\undefined%
    \setlength{\unitlength}{68.18424225bp}%
    \ifx\svgscale\undefined%
      \relax%
    \else%
      \setlength{\unitlength}{\unitlength * \real{\svgscale}}%
    \fi%
  \else%
    \setlength{\unitlength}{\svgwidth}%
  \fi%
  \global\let\svgwidth\undefined%
  \global\let\svgscale\undefined%
  \makeatother%
  \begin{picture}(1,0.38131992)%
    \lineheight{1}%
    \setlength\tabcolsep{0pt}%
    \put(0,0){\includegraphics[width=\unitlength,page=1]{m=3_a8_b4.pdf}}%
    \put(0.18467813,0.18332677){\makebox(0,0)[lt]{\lineheight{1.25}\smash{\begin{tabular}[t]{l}$\simeq$\end{tabular}}}}%
    \put(0.47800099,0.18332677){\makebox(0,0)[lt]{\lineheight{1.25}\smash{\begin{tabular}[t]{l}$\simeq$\end{tabular}}}}%
    \put(0.77132389,0.18332677){\makebox(0,0)[lt]{\lineheight{1.25}\smash{\begin{tabular}[t]{l}$\simeq$\end{tabular}}}}%
  \end{picture}%
\endgroup%

%% file: Monodromy.bib
@article{portilla_cuadrado_vanishing_2021,
	title = {Vanishing cycles, plane curve singularities and framed mapping class groups},
	volume = {25},
	pages = {3179--3228},
	number = {6},
	journaltitle = {Geom. Topol.},
	author = {Portilla Cuadrado, Pablo and Salter, Nick},
	date = {2021},
}

@article{lonne_fundamental_2007,
	title = {Fundamental group of discriminant complements of Brieskorn-Pham polynomials},
	volume = {345},
	pages = {93--96},
	number = {2},
	journaltitle = {C. R. Math. Acad. Sci. Paris},
	author = {Lönne, Michael},
	date = {2007},
	mrnumber = {2343559},
}

@article{baader_secondary_2021,
	title = {Secondary Braid Groups},
	journaltitle = {{arXiv}:2001.09098},
	author = {Baader, Sebastian and Lönne, Michael},
	date = {2021},
	eprinttype = {arxiv},
	eprint = {2001.09098},
}

@incollection{williams_braid_1988,
	title = {The braid index of an algebraic link},
	volume = {78},
	series = {Contemp. Math.},
	pages = {697--703},
	booktitle = {Braids (Santa Cruz, {CA}, 1986)},
	publisher = {Amer. Math. Soc., Providence, {RI}},
	author = {Williams, R. F.},
	date = {1988},
}

@article{hirasawa_visualization_2002,
	title = {Visualization of A'Campo's fibered links and unknotting operation},
	volume = {121},
	pages = {287--304},
	number = {1},
	journaltitle = {Topology Appl.},
	author = {Hirasawa, Mikami},
	date = {2002},
}

@article{acampo_sur_1973,
	title = {Sur la monodromie des singularités isolées d'hypersurfaces complexes},
	volume = {20},
	pages = {147--169},
	journaltitle = {Invent. Math.},
	author = {A'Campo, Norbert},
	date = {1973},
	mrnumber = {338436},
}

@article{johnson_spin_1980,
	title = {Spin structures and quadratic forms on surfaces},
	volume = {22},
	pages = {365--373},
	number = {2},
	journaltitle = {J. London Math. Soc. (2)},
	author = {Johnson, Dennis},
	date = {1980},
	mrnumber = {588283},
}

@article{couture_representative_2000,
	title = {Representative braids for links associated to plane immersed curves},
	volume = {9},
	pages = {1--30},
	number = {1},
	journaltitle = {J. Knot Theory Ramifications},
	author = {Couture, O. and Perron, B.},
	date = {2000},
	mrnumber = {1749500},
}

@article{acampo_real_1999,
	title = {Real deformations and complex topology of plane curve singularities},
	volume = {8},
	pages = {5--23},
	number = {1},
	journaltitle = {Ann. Fac. Sci. Toulouse Math. (6)},
	author = {A'Campo, Norbert},
	date = {1999},
	mrnumber = {1751447},
}

@article{cromwell_positive_1993,
	title = {Positive braids are visually prime},
	volume = {67},
	pages = {384--424},
	number = {2},
	journaltitle = {Proc. London Math. Soc. (3)},
	author = {Cromwell, Peter R.},
	date = {1993},
	mrnumber = {1226607},
}

@article{ishikawa_plumbing_2002,
	title = {Plumbing constructions of connected divides and the Milnor fibers of plane curve singularities},
	volume = {13},
	pages = {499--514},
	number = {4},
	journaltitle = {Indag. Math. (N.S.)},
	author = {Ishikawa, Masaharu},
	date = {2002},
	mrnumber = {2015834},
}

@article{liechti_genus_2020,
	title = {On the genus defect of positive braid knots},
	volume = {20},
	pages = {403--428},
	number = {1},
	journaltitle = {Algebr. Geom. Topol.},
	author = {Liechti, Livio},
	date = {2020},
	mrnumber = {4071377},
}

@article{goda_lissajous_2002,
	title = {Lissajous curves as A'Campo divides, torus knots and their fiber surfaces},
	volume = {25},
	pages = {485--491},
	number = {2},
	journaltitle = {Tokyo J. Math.},
	author = {Goda, Hiroshi and Hirasawa, Mikami and Yamada, Yuichi},
	date = {2002},
	mrnumber = {1948678},
}

@article{acampo_groupe_1975,
	title = {Le groupe de monodromie du déploiement des singularités isolées de courbes planes. I},
	volume = {213},
	pages = {1--32},
	journaltitle = {Math. Ann.},
	author = {A'Campo, Norbert},
	date = {1975},
	mrnumber = {377108},
}

@article{gusein-zade_intersection_1974,
	title = {Intersection matrices for certain singularities of functions of two variables},
	volume = {8},
	pages = {11--15},
	number = {1},
	journaltitle = {Funkcional. Anal. i Priložen.},
	author = {Guseın-Zade, S. M.},
	date = {1974},
	mrnumber = {0338437},
}

@article{randal-williams_homology_2014,
	title = {Homology of the moduli spaces and mapping class groups of framed, $r$-Spin and Pin surfaces},
	volume = {7},
	pages = {155--186},
	number = {1},
	journaltitle = {J. Topol.},
	author = {Randal-Williams, Oscar},
	date = {2014},
	mrnumber = {3180616},
}

@article{perron_groupe_1996,
	title = {Groupe de monodromie géométrique des singularités simples},
	volume = {306},
	pages = {231--245},
	number = {2},
	journaltitle = {Math. Ann.},
	author = {Perron, B. and Vannier, J. P.},
	date = {1996},
	mrnumber = {1411346},
}

@article{acampo_generic_1998,
	title = {Generic immersions of curves, knots, monodromy and Gordian number},
	pages = {151--169 (1999)},
	number = {88},
	journaltitle = {Inst. Hautes Études Sci. Publ. Math.},
	author = {A'Campo, Norbert},
	date = {1998},
	mrnumber = {1733329},
}

@article{labruere_generalized_1997,
	title = {Generalized braid groups and mapping class groups},
	volume = {6},
	pages = {715--726},
	number = {5},
	journaltitle = {J. Knot Theory Ramifications},
	author = {Labruère, C.},
	date = {1997},
	mrnumber = {1468903},
}

@article{gusein-zade_dynkin_1974,
	title = {Dynkin diagrams of the singularities of functions of two variables},
	volume = {8},
	pages = {23--30},
	number = {4},
	journaltitle = {Funkcional. Anal. i Priložen.},
	author = {Guseın-Zade, S. M.},
	date = {1974},
	mrnumber = {0430302},
}

@article{baader_checkerboard_2018,
	title = {Checkerboard graph monodromies},
	volume = {64},
	pages = {65--88},
	number = {1},
	journaltitle = {Enseign. Math.},
	author = {Baader, Sebastian and Lewark, Lukas and Liechti, Livio},
	date = {2018},
}

@article{wajnryb_artin_1999,
	title = {Artin groups and geometric monodromy},
	volume = {138},
	pages = {563--571},
	number = {3},
	journaltitle = {Invent. Math.},
	author = {Wajnryb, Bronislaw},
	date = {1999},
}

@book{farb_primer_2012,
	title = {A primer on mapping class groups},
	volume = {49},
	series = {Princeton Mathematical Series},
	pagetotal = {xiv+472},
	publisher = {Princeton University Press, Princeton, {NJ}},
	author = {Farb, Benson and Margalit, Dan},
	date = {2012},
}

@thesis{vilanova_positive_2020,
	location = {Bern},
	title = {Positive Hopf plumbed links with maximal signature},
	institution = {Universität Bern},
	type = {phdthesis},
	author = {Vilanova, Lucas},
	date = {2020-12},
}

@article{calderon_framed_2020,
	title = {Framed mapping class groups and the monodromy of strata of Abelian differentials},
	journaltitle = {{arXiv}:2002.02472},
	author = {Calderon, Aaron and Salter, Nick},
	date = {2020},
	eprinttype = {arxiv},
	eprint = {2002.02472},
}

@book{milnor_singular_1968,
	title = {Singular points of complex hypersurfaces},
	series = {Annals of Mathematics Studies, No. 61},
	pagetotal = {iii+122},
	publisher = {Princeton University Press, Princeton, N.J.; University of Tokyo Press, Tokyo},
	author = {Milnor, John},
	date = {1968},
	mrnumber = {0239612},
}

@book{brieskorn_plane_1986,
	title = {Plane algebraic curves},
	pagetotal = {vi+721},
	publisher = {Birkhäuser Verlag, Basel},
	author = {Brieskorn, Egbert and Knörrer, Horst},
	date = {1986},
}

@article{ryffel_curves_2023,
	title = {Curves intersecting in a circuit pattern},
	volume = {332},
	pages = {Paper No. 108522},
	journaltitle = {Topology Appl.},
	author = {Ryffel, Levi},
	date = {2023},
	mrnumber = {4575135},
}
